\documentclass[12pt, final]{amsart}

\usepackage{mathrsfs}
\usepackage{hyperref}
\usepackage[svgnames]{xcolor}
\hypersetup{colorlinks,breaklinks,
            linkcolor=DarkBlue,urlcolor=DarkBlue,
            anchorcolor=DarkBlue,citecolor=DarkBlue}
\usepackage{url}
\usepackage{euscript}
\usepackage[verbose,letterpaper,tmargin=1in,bmargin=1in,lmargin=1in,rmargin=1in]
{geometry}
\usepackage{amssymb}
\usepackage{pst-plot}
\usepackage{xy}

\newtheorem{theorem}{Theorem}[section]
\newtheorem{maintheorem}{Theorem}

\newtheorem{proposition}[theorem]{Proposition}

\newtheorem{lemma}[theorem]{Lemma}

\theoremstyle{definition}
\newtheorem{definition}[theorem]{Definition}
\newtheorem*{acknowledgement}{Acknowledgements}
\newtheorem{example}[theorem]{Example}
\newtheorem{remark*}[theorem]{}

\theoremstyle{remark}
\newtheorem{remark}[theorem]{Remark}

\newcommand{\C}{\mathbb C}

\newcommand{\N}{\mathbb N}

\newcommand{\tr}{\mathrm{tr}}

\DeclareMathOperator{\Hom}{Hom}

\newcommand{\mc}[1]{\mathcal{#1}}

\begin{document}

\title{Quantum invariant families of matrices in free probability}
\author[Stephen Curran]{Stephen Curran$^{(\dagger)}$}\thanks{$\dagger$ Research partially supported by an NSF postdoctoral fellowship}
\address{S.C.: Department of Mathematics, UCLA, Los Angeles, CA 90095, USA.}
\email{\href{mailto:curransr@math.ucla.edu}{\tt{curransr@math.ucla.edu}}}
\author[Roland Speicher]{Roland Speicher$^{(\ddagger)}$}\thanks{$\ddagger$ Research supported by a Discovery grant from NSERC}
\address{R.S.: Department of Mathematics and Statistics, Queen's University, Jeffery Hall, $\;$ Kingston, Ontario K7L 3N6, Canada,
{\rm and} Saarland University, FR 6.1 - Mathematik, Campus E 2.4, 66123 Saarbrucken, Germany.}
\email{\href{mailto:speicher@mast.queensu.ca}{\tt{speicher@mast.queensu.ca}} {\rm and} \href{mailto:speicher@math.uni-sb.de}{speicher@math.uni-sb.de}}
\begin{abstract}
We consider (self-adjoint) families of infinite matrices of noncommutative random variables such that the joint distribution of their entries is invariant under conjugation by a free quantum group.  For the free orthogonal and hyperoctahedral groups, we obtain complete characterizations of the invariant families in terms of an operator-valued $R$-cyclicity condition.  This is a surprising contrast with the Aldous-Hoover characterization of jointly exchangeable arrays.  
\end{abstract}

\maketitle

\section{Introduction}

A sequence $(X_1,X_2,\dotsc)$ of random variables is called \textit{exchangeable} (resp. \textit{rotatable}) if for each $n \in \N$ the joint distribution of $(X_1,\dotsc,X_n)$ is invariant under permutations (resp. orthogonal transformations).  De Finetti's celebrated theorem characterizes infinite exchangeable sequences as mixtures of i.i.d. sequences.  Likewise Freedman has characterized infinite rotatable sequences as mixtures of i.i.d. centered Gaussian sequences \cite{fre1}. 

Consider now an infinite symmetric matrix of random variables $(X_{ij})_{i,j \in \N}$, $X_{ij} = X_{ji}$.  Such a matrix is called \textit{jointly exchangeable} if $(X_{ij})_{i,j \in \N}$ has the same joint distribution as $(X_{\pi(i)\pi(j)})_{i,j \in \N}$ for any finite permutation $\pi$.  Equivalently, for each $n \in \N$ the joint distribution of the entries of $X_n = (X_{ij})_{1 \leq i,j \leq n}$ and $UX_nU^t$ agree for any $n \times n$ permutation matrix $U$.  There are two obvious examples of jointly exchangeable matrices: $(X_{ij})_{i \leq j}$ i.i.d. with $X_{ij} = X_{ji}$, and $X_{ij} = f(Y_i,Y_j)$ where $(Y_i)_{i \in \N}$ are i.i.d. and $f$ is symmetric in its arguments.  Further examples can be constructed from these, the most general being
\begin{equation*}
 X_{ij} = f(\alpha,\xi_i,\xi_j,\lambda_{ij})
\end{equation*}
where $\alpha$, $(\xi_i)_{i \in \N}$ and $(\lambda_{ij})_{i \leq j}$ are mutually independent and distributed uniformly on $[-1,1]$, $\lambda_{ij} = \lambda_{ji}$ and $f(a,\cdot,\cdot,d)$ is symmetric in its arguments for any fixed $a,d$.  A well known theorem of Aldous \cite{ald}, \cite{ald1} and Hoover \cite{hoover} states that any jointly exchangeable matrix can be represented in this way.  This result has recently reappeared in the contexts of limits of dense graphs \cite{dj}, classification of metric spaces with probability measures \cite{ver}, and hereditary properties of hypergraphs \cite{aus},\cite{aut}.  See the recent surveys by Aldous \cite{ald2},\cite{ald3} for further discussion and applications.  Likewise the \textit{jointly rotatable} matrices can be characterized as certain mixtures of Gaussian processes, see Kallenberg's text \cite{kal} for a thorough treatment of these and related results.

In \cite{ksp}, C. K\"{o}stler and the second author discovered that de Finetti's theorem has a natural analogue in free probability: an infinite sequence $(x_i)_{i \in \N}$ of noncommutative random variables is freely independent and identically distributed (with amalgamation over its tail algebra) if and only if for each $n \in \N$ the joint distribution of $(x_1,\dotsc,x_n)$ is ``invariant under quantum permutations''.  Here quantum permutation refers to Wang's \textit{free permutation group} $S_n^+$ \cite{wang2}, which is a compact quantum group in the sense of Woronowicz \cite{wor1}.  Likewise Freedman's characterization of rotatable sequences has a natural free analogue obtained by requiring invariance under Wang's \textit{free orthogonal group} $O_n^+$ \cite{wang1}, as shown by the first author in \cite{cur4}.  With T. Banica we have given a unified approach to de Finetti type theorems in the classical and free settings \cite{bcs2}, using the ``easiness'' formalism from \cite{bsp}.  See also \cite{cur3}, \cite{cur5}.

In this paper we consider matrices of noncommutative random variables $X = (x_{ij})_{1 \leq i,j \leq n}$ whose joint distribution is invariant under conjugation by $S_n^+,O_n^+,H_n^+$ or $B_n^+$, where $H_n^+$ is the \textit{free hyperoctahedral group} \cite{bbc2} and $B_n^+$ is the \textit{free bistochastic group} \cite{bsp}.  Given the analogy with the results of de Finetti and Freedman for sequences which are invariant under a free quantum group, one might expect to find a direct parallel with the Aldous-Hoover characterization.  However, the situation  is in fact quite different.  For example, matrices $X = (x_{ij})_{1 \leq i,j \leq n}$ with $(x_{ij})_{i \leq j}$ freely independent and identically distributed, and $x_{ji} = x_{ij}$, are not necessarily invariant under conjugation by $S_n^+$ (see Section \ref{sec:conclusion}).  Nevertheless, for $O_n^+$ and $H_n^+$ we are still able to obtain complete characterizations in terms of an operator-valued version of the $R$-cyclicity condition from \cite{nss2}.  Moreover, these characterizations extend naturally to invariant families of matrices $X_1,\dotsc,X_s$.  A surprising feature of these results is that they are ``matricial'' in nature, whereas the Aldous-Hoover characterization is often expressed as a statement about \textit{arrays}.

In the orthogonal case our main result is as follows (see Sections \ref{sec:prelim} and \ref{sec:rcyclic} for definitions and background):

\begin{maintheorem}\label{orthogonalcase}
Let $X_1,\dotsc,X_s$ be a family of infinite matrices, $X_r = (x_{ij}^{(r)})_{i,j \in \N}$, with entries in a W$^*$-probability space $(M,\varphi)$.  Assume that $M$ is generated as a von Neumann algebra by $\{x_{ij}^{(r)}:i,j \in \N, 1 \leq r \leq s\}$.  Assume moreover that the family is self-adjoint, in the sense that whenever $X$ is in the family, so is $X^*$.  Then the following conditions are equivalent:
\begin{enumerate}
 \item For each $n \in \N$, the joint distribution of the entries of $X_1,\dotsc,X_s$ is invariant under conjugation by $O_n^+$.
 \item There is a W$^*$-subalgebra $1 \in \mc B \subset M$ and a $\varphi$-preserving conditional expectation $E:M \to B$ such that the family $X_1,\dotsc,X_s$ is uniformly $R$-cyclic with respect to $E$.
\item There is a W$^*$-subalgebra $1 \in \mc B \subset M$ and a $\varphi$-preserving conditional expectation $E:M \to \mc B$, such that for each $n \in \N$, setting $X_r^{(n)} = (x_{ij}^{(r)})_{1 \leq i,j \leq n}$, we have $\{X_1^{(n)},\dotsc,X_r^{(n)}\} \subset M_n(M)$ is freely independent from $M_n(\mc B)$ with amalgamation over $\mc B$.
\end{enumerate}
\end{maintheorem}

The equivalence of (2) and (3) is well known in the case $\mc B = \C$, see e.g. \cite{ns}.  We will prove this for general $\mc B$ in Section \ref{sec:rcyclic}.  One feature of operator-valued uniformly $R$-cyclic families $X_1,\dotsc,X_s$ is that the (operator-valued) joint distribution of their entries is completely determined by that of $(x_{11}^{(1)},\dotsc,x_{11}^{(s)})$.  It is therefore natural to wonder what distributions may arise in this way.  We will show that these are exactly the operator-valued distributions which are freely infinitely divisible (known in the case $\mc B = \C$, see e.g. \cite{ns}).
\begin{maintheorem}\label{infdivis}
Let $\mc A$ be a unital C$^*$-algebra, $1 \in \mc B \subset \mc A$ a C$^*$-subalgebra and $E:\mc A \to \mc B$ a faithful, completely positive conditional expectation.  Let $y_1,\dotsc,y_s \in \mc A$, then the following conditions are equivalent:
\begin{enumerate}
 \item There is a unital C$^*$-algebra $\mc A'$, a unital inclusion $\mc B \hookrightarrow \mc A'$, a faithful completely positive conditional expectation $E':\mc A' \to \mc B$, and a family $\{x_{ij}^{(r)}:i,j \in \N, 1 \leq r \leq s\} \subset \mc A'$ such that
\begin{itemize}
\item $(x_{11}^{(1)},\dotsc,x_{11}^{(s)})$ has the same $\mc B$-valued distribution as $(y_1,\dotsc,y_s)$.
\item $X_1,\dotsc,X_s$ forms a $\mc B$-valued uniformly $R$-cyclic family, where $X_r = (x_{ij}^{(r)})_{i,j \in \N}$ for $1 \leq r \leq s$.
\end{itemize}
\item The $\mc B$-valued joint distribution of $(y_1,\dotsc,y_s)$ is freely infinitely divisible, i.e. for each $n \in \N$ there exists a unital C$^*$-algebra $\mc A_n$, a unital inclusion $\mc B \hookrightarrow \mc A_n$, a faithful completely positive conditional expectation $E_n:\mc A_n \to \mc B$, and a family $\{y_r^{(i)}: 1 \leq i \leq n, 1 \leq r \leq s\}$ such that 
\begin{itemize}
\item The families $\{y_1^{(1)},\dotsc,y_s^{(1)}\},\dotsc,\{y_{1}^{(n)},\dotsc,y_s^{(n)}\}$ are freely independent with respect to $E_n$.
\item The $\mc B$-valued joint distribution of $(y_1^{(i)},\dotsc,y_s^{(i)})$ does not depend on $1 \leq i \leq n$.
\item $(y_1,\dotsc,y_s)$ has the same $\mc B$-valued distribution as $(y'_1,\dotsc,y'_s)$, where $y'_r = y_r^{(1)} + \dotsb + y_r^{(n)}$ for $1 \leq r \leq s$.
\end{itemize}
\end{enumerate}
  
\end{maintheorem}

For self-adjoint families of infinite matrices which are invariant under conjugation by the free hyperoctahedral group, our main result is as follows:
\begin{maintheorem}\label{hyperoctahedralcase}
Let $X_1,\dotsc,X_s$ be a family of infinite matrices, $X_r = (x_{ij}^{(r)})_{i,j \in \N}$, with entries in a W$^*$-probability space $(M,\varphi)$.  Assume that $M$ is generated as a von Neumann algebra by $\{x_{ij}^{(r)}:i,j \in \N, 1 \leq r \leq s\}$.  Assume moreover that the family is self-adjoint, in the sense that whenever $X$ is in the family, so is $X^*$.  Then the following conditions are equivalent:
\begin{enumerate}
 \item For each $n \in \N$, the joint distribution of the entries of $X_1,\dotsc,X_s$ is invariant under conjugation by $H_n^+$.
 \item There is a W$^*$-subalgebra $1 \in \mc B \subset M$ and a $\varphi$-preserving conditional expectation $E:M \to B$ such that the family $X_1,\dotsc,X_s$ is $R$-cyclic with respect to $E$, and its determining series is invariant under quantum permutations.
\end{enumerate}
\end{maintheorem}

The $R$-cyclicity condition appearing in (2) is equivalent to freeness of the family $X_1,\dotsc,X_s$ from $M_n(\mc B)$, but with amalgamation now over the algebra of diagonal matrices with entries in $\mc B$.  This is known in the case $\mc B = \C$ from \cite{nss}, we prove this for general $\mc B$ in Section 3.

The situation for $S_n^+$ and $B_n^+$-invariant matrices appears to be much more complicated.  In particular, invariant matrices need not be $R$-cyclic.  For example, constant matrices $x_{ij} = \alpha$ are invariant under conjugation by $B_n^+$ (and hence $S_n^+$), but are not $R$-cyclic if $\alpha \neq 0$.  The contrast between $S_n^+$ and $H_n^+$-invariant matrices is surprising, given the similar characterizations of invariant sequences (see \cite{bcs2}).  Moreover, it follows from the Aldous-Hoover characterization that an infinite symmetric matrix of classical random variables is invariant under conjugation by the hyperoctahedral group if and only if it has a representation of the form $f(\alpha,\xi_i,\xi_j,\lambda_{ij})$, as for jointly exchangeable matrices, with the only additional condition being that $f$ is an odd function of each of its entries.  We will give some partial results for $S_n^+$ and $B_n^+$-invariant families in Section \ref{sec:qinv}, but leave the classification problem open.  We will discuss $S_n^+$-invariant matrices further in Section \ref{sec:conclusion}.

Our paper is organized as follows.  Section \ref{sec:prelim} contains preliminaries, here we recall the basic concepts from free probability.  We also recall some basic notions and results from \cite{bsp} on ``free'' quantum groups.  In Section \ref{sec:rcyclic} we develop the basic theory of operator-valued $R$-cyclic matrices, and prove Theorem \ref{infdivis}.  This generalizes the results from \cite{nss}, and may be of independent interest.  In Section \ref{sec:qinv} we study families of matrices of noncommutative random variables which are invariant under conjugation by a free quantum group.  We give a combinatorial description of invariant families of finite matrices in Theorem \ref{easyinv}.  We then give a general formula for operator-valued moment and cumulant functionals of the entries of a self-adjoint family of infinite matrices which is invariant under conjugation by a free quantum group.  In Sections \ref{sec:o+case} and \ref{sec:h+case}, we further analyze this formula in the free orthogonal and hyperoctahedral cases, and prove Theorems \ref{orthogonalcase} and \ref{hyperoctahedralcase}.  Section \ref{sec:conclusion} contains concluding remarks, including further discussion of $S_n^+$-invariant matrices.

\begin{acknowledgement}
We would like to thank T. Banica and D. Shlyakhtenko for several useful discussions.  Part of this work was done while S.C. was completing his Ph.D. at UC Berkeley, and he would like to thank his thesis supervisor, Dan-Virgil Voiculescu, for his guidance and support.  S.C. would also like to thank Queen's University, where another part of this work was done.
\end{acknowledgement}

\section{Notations and Preliminaries}\label{sec:prelim}

\subsection{Free probability}  We begin by recalling the basic notions of noncommutative probability spaces and distributions of random variables.

\begin{definition}\hfill
\begin{enumerate}
 \item A \textit{noncommutative probability space} is a pair $(\mc A, \varphi)$, where $\mc A$ is a unital algebra over $\C$ and $\varphi:\mc A \to \C$ is a linear functional such that $\varphi(1) = 1$.  Elements in $\mc A$ will be called \textit{random variables}.
\item A W$^*$-probability space $(M,\varphi)$ is a von Neumann algebra $M$ together with a faithful, normal state $\varphi$.  We will not assume that $\varphi$ is tracial.
\end{enumerate}

\end{definition}

The \textit{joint distribution} of a family $(x_i)_{i \in I}$ of random variables in a
noncommutative probability space $(\mc A,\varphi)$ is the collection of \textit{joint
moments}
\begin{equation*}
 \varphi(x_{i_1}\dotsb x_{i_k})
\end{equation*}
for $k \in \N$ and $i_1,\dotsc,i_k \in I$.  This is encoded in the linear
functional $\varphi_x:\C\langle t_i| i \in I \rangle \to \C$ determined by
\begin{equation*}
 \varphi_x(p) = \varphi(p(x))
\end{equation*}
for $p \in \C \langle t_i| i \in I \rangle$, where $p(x)$ means of course to replace
$t_i$ by $x_i$ for each $i \in I$.  Here $\C \langle t_i | i \in I \rangle$ denotes the algebra of polynomials in \textit{non-commuting} indeterminates.

These definitions have natural ``operator-valued'' extensions given by replacing $\C$ by
a more general algebra of scalars, which we now recall.

\begin{definition}
An \textit{operator-valued probability space} $(\mc A, E:\mc A \to \mc B)$ consists of a
unital algebra $\mc A$, a subalgebra $1 \in \mc B \subset \mc A$, and a conditional
expectation $E:\mc A \to \mc B$, i.e., $E$ is a linear map such that $E[1] = 1$ and
\begin{equation*}
 E[b_1ab_2] = b_1E[a]b_2
\end{equation*}
for all $b_1,b_2 \in \mc B$ and $a \in \mc A$.
\end{definition}

The \textit{$\mc B$-valued joint distribution} of a family $(x_i)_{i \in I}$ of random
variables in an operator-valued probability space $(\mc A,E:\mc A \to \mc B)$ is the
collection of \textit{$\mc B$-valued joint moments}
\begin{equation*}
 E[b_0x_{i_1}\dotsb x_{i_k}b_k]
\end{equation*}
for $k \in \N$, $i_1,\dotsc,i_k \in I$ and $b_0,\dotsc,b_k \in \mc B$.

\begin{definition}
Let $(\mc A, E:\mc A \to \mc B)$ be an operator-valued probability space, and let $(\mc
A_i)_{i \in I}$ be a collection of subalgebras $\mc B \subset \mc A_i \subset A$.  The
algebras are said to be \textit{free with amalgamation over $\mc B$}, or \textit{freely
independent with respect to $E$}, if
\begin{equation*}
 E[a_1\dotsb a_k] = 0
\end{equation*}
whenever $E[a_j] = 0$ for $1 \leq j \leq k$ and $a_j \in \mc A_{i_j}$ with $i_{j} \neq
i_{j+1}$ for $1 \leq j < k$.

We say that subsets $\Omega_i \subset \mc A$ are free with amalgamation over $\mc B$ if
the subalgebras $\mc A_i$ generated by $\mc B$ and $\Omega_i$ are freely independent with
respect to $E$.
\end{definition}

\begin{remark}
Voiculescu first defined freeness with amalgamation, and developed its basic theory in
\cite{voi0}.  Freeness with amalgamation also has a rich combinatorial structure,
developed in \cite{sp1}, which we now recall.  For further information on the
combinatorial theory of free probability, the reader is referred to the text \cite{ns}.
\end{remark}

\begin{definition}\hfill
\begin{enumerate}
\item A \textit{partition} $\pi$ of a set $S$ is a collection of disjoint, non-empty sets $V_1,\dotsc,V_r$ such that $V_1 \cup \dotsb \cup V_r = S$.  $V_1,\dotsc,V_r$ are called the \textit{blocks} of $\pi$, and we set $|\pi| = r$. If $s,t \in S$ are in the same block of $\pi$, we write $s \sim_\pi t$. The collection of partitions of $S$ will be denoted $\mc P(S)$, or in the case that $S =\{1,\dotsc,k\}$ by $\mc P(k)$.
\item Given $\pi,\sigma \in \mc P(S)$, we say that $\pi \leq \sigma$ if each
block of $\pi$ is contained in a block of $\sigma$.
There is a least element of $\mc P(S)$ which is larger than both $\pi$ and $\sigma$,
which we denote by $\pi \vee \sigma$.  Likewise there is a greatest element which is smaller than both $\pi$ and $\sigma$, denoted $\pi \wedge \sigma$.  
\item If $S$ is ordered, we say that $\pi \in \mc P(S)$ is \textit{non-crossing} if whenever $V,W$ are blocks of $\pi$ and $s_1 < t_1 < s_2 < t_2$ are such that $s_1,s_2 \in V$ and $t_1,t_2 \in W$, then $V = W$.  The non-crossing partitions can also be defined recursively, a partition $\pi \in \mc P(S)$ is non-crossing if and only if it has a block $V$ which is an interval, such that $\pi \setminus V$ is a non-crossing partition of $S \setminus V$.  The set of non-crossing partitions of $S$ is denoted by $NC(S)$, or by $NC(k)$ in the case that $S = \{1,\dotsc,k\}$.
\item $NC_h(k)$ will denote the collection of non-crossing partitions of $\{1,\dotsc,k\}$ for which each block contains an even number of elements.  Likewise $NC_2(k)$ will denote the non-crossing partitions for which each block contains exactly two elements.
\item  Given $i_1,\dotsc,i_k$ in some index set $I$, we denote by $\ker \mathbf i$ the element of $\mc P(k)$ whose blocks are the equivalence classes of the relation
\begin{equation*}
 s \sim t \Leftrightarrow i_s= i_t.
\end{equation*}
Note that if $\pi \in \mc P(k)$, then $\pi \leq \ker \mathbf i$ is equivalent to the
condition that whenever $s$ and $t$ are in the same block of $\pi$, $i_s$ must equal
$i_t$.
\item $0_k$ and $1_k$ will denote the smallest and largest partitions in $NC(k)$, i.e. $0_k$ has $k$ blocks with one element each, and $1_k$ has one block containing $1,\dotsc,k$.
\end{enumerate}
\end{definition}

\begin{definition} Let $(\mc A, E:\mc A \to \mc B)$ be an operator-valued probability space.
\begin{enumerate}
 \item A \textit{$\mc B$-functional} is a $n$-linear map $\rho:\mc A^n \to \mc B$ such that
\begin{equation*}
 \rho(b_0a_1b_1,a_2b_2,\dotsc,a_nb_n) = b_0\rho(a_1,b_1a_2,\dotsc,b_{n-1}a_n)b_n
\end{equation*}
for all $b_0,\dotsc,b_n \in \mc B$ and $a_1,\dotsc,a_n \in \mc A$.  Equivalently, $\rho$
is a linear map from $\mc A^{\otimes_B n}$ to $\mc B$, where the tensor product is taken
with respect to the obvious $\mc B$-$\mc B$-bimodule structure on $\mc A$.
\item For each $k \in \N$, let $\rho^{(k)}:\mc A^k \to \mc B$ be a $\mc B$-functional.  For $n \in \N$ and $\pi \in NC(n)$, we define a $\mc B$-functional $\rho^{(\pi)}:\mc A^n \to \mc B$ recursively as follows:  If $\pi = 1_n$ is the partition containing only one block, we set $\rho^{(\pi)} = \rho^{(n)}$.  Otherwise let $V = \{l+1,\dotsc,l+s\}$ be an interval of $\pi$ and define
\begin{equation*}
 \rho^{(\pi)}[a_1,\dotsc,a_n] = \rho^{(\pi \setminus V)}[a_1,\dotsc,a_l\rho^{(s)}(a_{l+1},\dotsc,a_{l+s}),a_{l+s+1},\dotsc,a_n]
\end{equation*}
for $a_1,\dotsc,a_n \in \mc A$.
\end{enumerate}
\end{definition}

\begin{example}
Let $(\mc A,E:\mc A \to \mc B)$ be an operator-valued probability space, and for $k \in
\N$ let $\rho^{(k)}:\mc A^k \to \mc B$ be a $\mc B$-functional as above.  If
 \begin{equation*}
\pi = \{\{1,8,9,10\},\{2,7\},\{3,4,5\}, \{6\}\} \in NC(10),
\end{equation*}
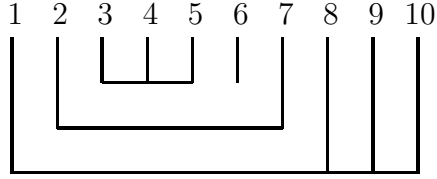
\begin{equation*}
 \setlength{\unitlength}{0.6cm} \begin{picture}(9,4)\thicklines \put(0,0){\line(0,1){3}}
\put(0,0){\line(1,0){9}} \put(9,0){\line(0,1){3}} \put(8,0){\line(0,1){3}}
\put(7,0){\line(0,1){3}} \put(1,1){\line(1,0){5}} \put(1,1){\line(0,1){2}}
\put(6,1){\line(0,1){2}} \put(2,2){\line(1,0){2}} \put(2,2){\line(0,1){1}}
\put(3,2){\line(0,1){1}} \put(4,2){\line(0,1){1}} \put(5,2){\line(0,1){1}}
\put(-0.1,3.3){1} \put(0.9,3.3){2} \put(1.9,3.3){3} \put(2.9,3.3){4} \put(3.9,3.3){5}
\put(4.9,3.3){6} \put(5.9,3.3){7} \put(6.9,3.3){8} \put(7.9,3.3){9} \put(8.7,3.3){10}
\end{picture}
\end{equation*}
then the corresponding $\rho^{(\pi)}$ is given by
\begin{equation*}
 \rho^{(\pi)}[a_1,\dotsc,a_{10}] = \rho^{(4)}(a_1\cdot \rho^{(2)}(a_2\cdot \rho^{(3)}(a_3,a_4,a_5),\rho^{(1)}(a_6)\cdot a_7),a_8,a_9,a_{10}).
\end{equation*}
\end{example}

\begin{definition}
Let $(\mc A, E:\mc A \to \mc B)$ be an operator-valued probability space.
\begin{enumerate}
\item For $k \in \N$, define the \textit{$B$-valued moment functions} $E^{(k)}:\mc A^k \to \mc B$ by
\begin{equation*}
 E^{(k)}[a_1,\dotsc,a_k] = E[a_1\dotsb a_k].
\end{equation*}

\item The \textit{operator-valued free cumulants} $\kappa_E^{(k)}:\mc A^k \to \mc B$ are the $\mc B$-functionals defined by the \textit{moment-cumulant formula}:
\begin{equation*}
 E[a_1\dotsb a_n] = \sum_{\pi \in NC(n)} \kappa_E^{(\pi)}[a_1,\dotsc,a_n]
\end{equation*}
for $n \in \N$ and $a_1,\dotsc,a_n \in \mc A$.
\end{enumerate}

\end{definition}

Note that the right hand side of the moment-cumulant formula is equal to
$\kappa_E^{(n)}[a_1,\dotsc,a_n]$ plus lower order terms and hence can be
solved recursively for $\kappa_E^{(n)}$.  In fact the cumulant functions can be solved
from the moment functions by the following formula from \cite{sp1}: for each $n \in \N$,
$\pi \in NC(n)$ and $a_1,\dotsc,a_n \in \mc A$,
\begin{equation*}
 \kappa_E^{(\pi)}[a_1,\dotsc,a_n] = \sum_{\substack{\sigma \in NC(n)\\ \sigma \leq \pi}} \mu_n(\sigma,\pi)E^{(\sigma)}[a_1,\dotsc,a_n],
\end{equation*}
where $\mu_n$ is the \textit{M\"{o}bius function} on the partially ordered set $NC(n)$.  $\mu_n$ is characterized by the relations
\begin{equation*}
 \sum_{\substack{\tau \in NC(n)\\ \sigma \leq \tau \leq \pi}} \mu_n(\sigma,\tau) = \delta_{\sigma,\pi} = \sum_{\substack{\tau \in NC(n)\\ \sigma \leq \tau \leq \pi}} \mu_n(\tau,\pi)
\end{equation*}
for any $\sigma \leq \pi$ in $NC(n)$, and $\mu(\sigma,\pi) = 0$ if $\sigma \not\leq \pi$, see \cite{ns}.

The key relation between operator-valued free cumulants and freeness with amalgamation is
that freeness can be characterized in terms of the ``vanishing of mixed cumulants''.
\begin{theorem}[\cite{sp1}]
Let $(\mc A, E:\mc A \to \mc B)$ be an operator-valued probability space, and let $(\mc
A_i)_{i \in I}$ be a collection of subalgebras $\mc B \subset \mc A_i \subset \mc A$.
Then the family $(\mc A_i)_{i \in I}$ is free with amalgamation over $\mc B$ if and only
if
\begin{equation*}
 \kappa_E^{(\pi)}[a_1,\dotsc,a_n] = 0
\end{equation*}
whenever $a_j \in \mc A_{i_j}$ for $1 \leq j \leq n$ and $\pi \in NC(n)$ is such that
$\pi \not\leq \ker \mathbf i$. \qed
\end{theorem}

\subsection{``Fattening'' of noncrossing partitions}  A theme in this paper will be relating noncrossing partitions of $k$ points with those of $2k$ points.  The basic operation we will use is the ``fattening'' procedure, which gives a bijection $\pi \mapsto \widetilde \pi$ from $NC(k)$ to $NC_2(2k)$.  Let us now recall this procedure, along with some related operations on partitions.

\begin{definition}\hfill
\begin{enumerate}
\item Given $\pi \in NC(k)$, we define $\widetilde \pi \in NC_2(2k)$ as follows:  For each block $V = (i_1,\dotsc,i_s)$ of $\pi$, we add to $\widetilde \pi$ the pairings $(2i_1-1,2i_s), (2i_1,2i_2-1),\dotsc$, $(2i_{s-1},2i_s-1)$.
\item Given $\pi \in NC(k)$, we define $\hat \pi \in NC(2k)$ by
partitioning the $k$ pairs $(1,2),(3,4),\dotsc$, $(2k-1,2k)$ according to $\pi$.
\item Given $\pi \in \mc P(k)$, let $\overleftarrow{\pi}$ denote the partition obtained by shifting $k$ to $k-1$ for $1 < k \leq m$ and sending $1$ to $m$, i.e.,
\begin{equation*}
s \sim_{\overleftarrow{\pi}} t \qquad\Longleftrightarrow \qquad(s + 1) \sim_\pi (t+1),
\end{equation*}
where we count modulo $k$ on the right hand side.  Likewise we define $\overrightarrow{\pi}$ in the obvious way.
\end{enumerate}
\end{definition}

\begin{example} Let us demonstrate these operations for $\pi=\{\{1,4,5\},\{2,3\}$, $\{6\}\}$.

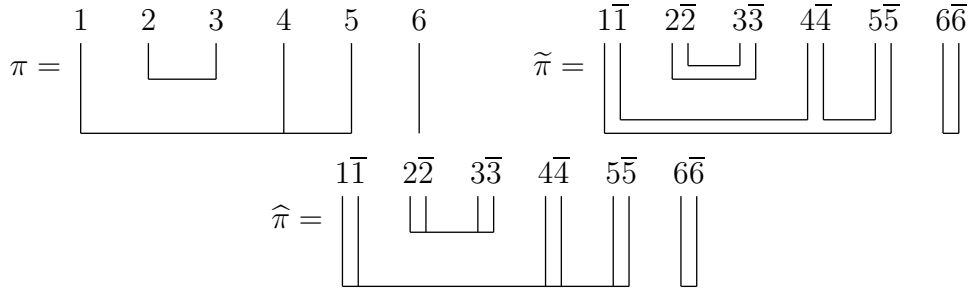
\begin{center}
\begin{pspicture}(0,0)(6,2)
{\psset{xunit=.3cm,yunit=.6cm,linewidth=.5pt}
\psline(0,2)(0,0)\psline(0,0)(12,0) \psline(3,1.2)(6,1.2)
 \psline(12,2)(12,0)
 \psline(3,2)(3,1.2) \psline(6,2)(6,1.2)
\psline(9,2)(9,0)\psline(15,0)(15,2)\uput[u](0,2){1}\uput[u](3,2){2}\uput[u](6,2){3}
\uput[u](9,2){4}\uput[u](12,2){5} \uput[u](15,2){6} \uput[u](-2,1){$\pi=$}}
\end{pspicture}
\qquad
\begin{pspicture}(0,0)(4.8,2)
{\psset{xunit=.3cm,yunit=.6cm,linewidth=.5pt} \psline(0,2)(0,0)\psline(0,0)(12.7,0)
\psline(3,1.2)(6.7,1.2)\psline(3.7,2)(3.7,1.5)
\psline(9.7,2)(9.7,.3)\psline(9.7,.3)(12,.3) \psline(12,2)(12,0.3)
 \psline(.7,2)(.7,.3)\psline(.7,.3)(9,.3)\psline(6.7,1.2)(6.7,2)
\psline(3.7,1.5)(6,1.5) \psline(3,2)(3,1.2) \psline(6,2)(6,1.5)
\psline(9,2)(9,0.3)\psline(15,0)(15,2) \psline(15.7,0)(15.7,2)\psline(15.7,0)(15,0)
\psline(12.7,0)(12.7,2) \uput[u](0,2){1} \uput[u](.7,2){$\overline
1$}\uput[u](3,2){2}\uput[u](3.7,2){$\overline 2$}\uput[u](6,2){3}\uput[u](6.7,2){$\overline 3$}
\uput[u](9,2){4}\uput[u](9.7,2){$\overline
4$}\uput[u](12,2){5}\uput[u](12.7,2){$\overline
5$}\uput[u](15,2){6}\uput[u](15.7,2){$\overline 6$} \uput[u](-2,1){$\widetilde\pi=$}}
\end{pspicture}\\
\begin{pspicture}(0,0)(4.8,2)
{\psset{xunit=.3cm,yunit=.6cm,linewidth=.5pt} \psline(0,2)(0,0)\psline(0,0)(12.7,0)
\psline(3,1.2)(6.7,1.2)\psline(3.7,2)(3.7,1.2)
\psline(9.7,2)(9.7,0) \psline(12,2)(12,0)
 \psline(.7,2)(.7,0)\psline(6.7,1.2)(6.7,2)
 \psline(3,2)(3,1.2) \psline(6,2)(6,1.2)
\psline(9,2)(9,0)\psline(15,0)(15,2) \psline(15.7,0)(15.7,2)\psline(15.7,0)(15,0)
\psline(12.7,0)(12.7,2) \uput[u](0,2){1} \uput[u](.7,2){$\overline
1$}\uput[u](3,2){2}\uput[u](3.7,2){$\overline 2$}\uput[u](6,2){3}\uput[u](6.7,2){$\overline 3$}
\uput[u](9,2){4}\uput[u](9.7,2){$\overline
4$}\uput[u](12,2){5}\uput[u](12.7,2){$\overline
5$}\uput[u](15,2){6}\uput[u](15.7,2){$\overline 6$} \uput[u](-2,1){$\widehat\pi=$}}
\end{pspicture}
\end{center}

\end{example}

There is a simple description of the inverse of the fattening procedure: it sends $\sigma \in NC_2(2k)$ to the
partition $\tau \in NC(k)$ such that $\sigma \vee \hat 0_k = \hat \tau$.  Thus we have
\begin{equation*}
\hat \pi=\widetilde \pi\vee \hat 0_k
\end{equation*}
for $\pi \in NC(k)$.  Note also that $\hat 0_k=\widetilde 0_k$ and that $\hat 1_k=1_{2k}$.

Let us now introduce two more operations on partitions.
\begin{definition} \hfill
\begin{enumerate}
 \item Given $\pi,\sigma \in NC(k)$, we define $\pi \wr \sigma \in \mc P(2k)$ to be the partition obtained by partitioning the odd numbers $\{1,3,\dotsc,2k-1\}$ according to $\pi$ and the even numbers $\{2,4,\dotsc,2k\}$ according to $\sigma$.
\item Let $\pi \in NC(k)$.  The \textit{Kreweras complement} $K(\pi)$ is the largest partition in $NC(k)$ such that $\pi \wr K(\pi)$ is noncrossing.
\end{enumerate}
\end{definition}

\begin{remark}
The Kreweras complement is in fact a lattice anti-isomorphism of $NC(k)$ with itself, and plays an important role in the combinatorics of free probability.  As we recall below, there are nice relations between the Kreweras complement and the fattening procedure.
\end{remark}

\begin{example}
If $\pi = \{\{1,5\},\{2,3,4\},\{6,8\},\{7\}\}$ then $K(\pi) =
\{\{1,4\},\{2\},\{3\},\{5,8\}$, $\{6,7\}\}$, which can be seen follows:
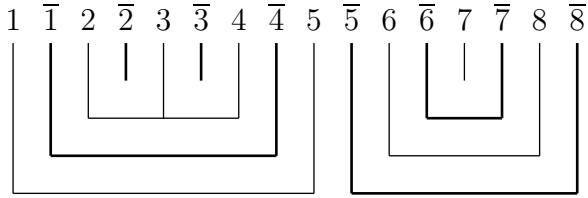
\begin{center}
\begin{pspicture}(0,0)(8,2.5)
{\psset{xunit=.5cm,yunit=1cm,linewidth=.5pt} \psline(0,2)(0,0)\psline(0,0)(8,0)\psline(8,0)(8,2)
\psline(2,2)(2,1)\psline(2,1)(6,1)\psline(6,1)(6,2)\psline(4,1)(4,2)
\psline(10,2)(10,.5)\psline(10,.5)(14,.5)\psline(14,.5)(14,2) \psline(12,2)(12,1.5)}
{\psset{xunit=.5cm,yunit=1cm,linewidth=1pt} \psline(1,2)(1,.5)\psline(1,.5)(7,.5)\psline(7,.5)(7,2)
\psline(3,2)(3,1.5) \psline(5,2)(5,1.5)
\psline(9,2)(9,0)\psline(9,0)(15,0)\psline(15,0)(15,2)
\psline(11,2)(11,1)\psline(11,1)(13,1)\psline(13,1)(13,2) \uput[u](0,2){1}
\uput[u](1,2){$\overline 1$}\uput[u](2,2){2}\uput[u](3,2){$\overline
2$}\uput[u](4,2){3}\uput[u](5,2){$\overline 3$} \uput[u](6,2){4}\uput[u](7,2){$\overline
4$}\uput[u](8,2){5}\uput[u](9,2){$\overline 5$}\uput[u](10,2){6}\uput[u](11,2){$\overline
6$} \uput[u](12,2){7}\uput[u](13,2){$\overline
7$}\uput[u](14,2){8}\uput[u](15,2){$\overline 8$}}
\end{pspicture}
\end{center}

\end{example}

The following key lemma connecting these operations was proved in \cite{cs1}.  Note that (1) is a generalization of
\begin{equation*}
 \widetilde{K(0_k)} = \widetilde{1_k} = \overleftarrow{\widetilde{0_k}},
\end{equation*}
and (2) is a generalization of the relation
\begin{equation*}
 K(\widetilde 0_k \vee \widetilde \pi) = K(\hat \pi) = 0_k \wr K(\pi)
\end{equation*}
for $\pi \in NC(k)$, both of which are clear from the definitions.

\begin{lemma}[\cite{cs1}]\label{fatfacts}\hfill
\begin{enumerate}
\item If $\pi \in NC(k)$ then
\begin{equation*}
 \widetilde{K(\pi)} = \overleftarrow{\widetilde \pi}.
\end{equation*}

\item If $\sigma,\pi \in NC(k)$ and $\sigma \leq \pi$, then $\widetilde \sigma \vee \widetilde \pi \in NC_h(2k)$ and
\begin{equation*}
 K(\widetilde \sigma \vee \widetilde \pi) = \sigma \wr K(\pi).
\end{equation*}
\end{enumerate}
\qed
\end{lemma}

\subsection{Free quantum groups}\label{sec:easy}  We now briefly recall some notions and results from \cite{bsp}.  

\begin{definition}[\cite{wor1}]
An \textit{orthogonal Hopf algebra} is a unital C$^*$-algebra $A$ generated by self-adjoint elements $\{u_{ij}:1 \leq i,j \leq n\}$, such that the following conditions hold:
\begin{enumerate}
 \item The inverse of $u = (u_{ij}) \in M_n(A)$ is the transpose $u^t = (u_{ji})$.
 \item $\Delta(u_{ij}) = \sum_{k} u_{ik} \otimes u_{kj}$ determines a morphism $\Delta:A \to A \otimes A$.
 \item $\epsilon(u_{ij}) = \delta_{ij}$ defines a morphism $\epsilon:A \to \C$.
\item $S(u_{ij}) = u_{ji}$ defines a morphism $S:A \to A^{op}$.  
\end{enumerate}
\end{definition}

It follows from the definitions that $\Delta, \epsilon, S$ satisfy the usual Hopf algebra axioms.  The motivating example is $C(G)$ where $G \subset O_n$ is a compact group of orthogonal matrices, here $u_{ij}$ are the coordinate functions sending $g \in G$ to its $(i,j)$-entry $g_{ij}$.

In fact any commutative orthogonal Hopf algebra is $C(G)$ for a compact group $G \subset O_n$.  We will therefore use the heuristic notation ``$A = C(G)$'', where $G$ is a \textit{compact orthogonal quantum group}.  Of course if $A$ is noncommutative then $G$ cannot exist as a concrete object, and all statements about $G$ must be interpreted in terms of the Hopf algebra $A$.

We will be mostly interested in the following examples, constructed in \cite{wang1},\cite{wang2},\cite{bbc2},\cite{bsp}.

\begin{definition}\hfill
\begin{enumerate}
 \item $C(O_n^+)$ is the universal C$^*$-algebra generated by self-adjoint $\{u_{ij}:1 \leq i,j \leq n\}$, such that $u = (u_{ij}) \in M_n(C(O_n^+))$ is orthogonal.
\item $C(S_n^+)$ is the universal C$^*$-algebra generated by projections $\{u_{ij}:1 \leq i,j \leq n\}$, such that the sum along any row or column of $u = (u_{ij}) \in M_n(C(S_n^+))$ is the identity.
\item $C(H_n^+)$ is the universal C$^*$-algebra generated by self-adjoint $\{u_{ij}:1 \leq i,j \leq n\}$ such that $u = (u_{ij}) \in M_n(C(H_n^+))$ is orthogonal and $u_{ik}u_{il} = 0 = u_{kj}u_{lj}$ if $k \neq l$.
\item $C(B_n^+)$ is the universal C$^*$-algebra generated by self-adjoint $\{u_{ij}:1 \leq i,j \leq n\}$ such that $u = (u_{ij}) \in M_n(C(B_n^+))$ is orthogonal and the sum along any row or column of $u$ is the identity.
\end{enumerate}
\end{definition}

In each case the existence of the Hopf algebra morphisms follows from the defining universal properties.  Note that the we have the following inclusions:
\begin{equation*}
\begin{matrix}
B_n^+&\subset&O_n^+\cr
&&\cr
\cup&&\cup\cr
&&\cr
S_n^+&\subset&H_n^+
\end{matrix}
\end{equation*}

Our interest in these quantum groups are that they are ``free versions'' of the classical orthogonal, permutation, hyperoctahedral and bistochastic groups.  To make this notion precise, it is best to look at the representation theory of these quantum groups.

Let $S_n \subset G \subset O_n^+$ be a compact orthogonal quantum group and let $u,v$ be the fundamental representations of $G,S_n$ on $\C^n$, respectively.  By functoriality, the space $\Hom(u^{\otimes k},u^{\otimes l})$ of intertwining operators is contained in $\Hom(v^{\otimes k}, v^{\otimes l})$ for any $k,l$.  But the Hom-spaces for $v$ are well-known: they are spanned by operators $T_\pi$ with $\pi$ belonging to the set $\mc P(k,l)$ of partitions between $k$ upper and $l$ lower points.  Explicitly, if $e_1,\dotsc,e_n$ denotes the standard basis of $\C^n$, then the formula for $T_\pi$ is given by
\begin{equation*}
 T_\pi(e_{i_1} \otimes \dotsb \otimes e_{i_k}) = \sum_{j_1,\dotsc,j_l} \delta_\pi\begin{pmatrix} i_1 \dotsb i_k\\ j_1\dotsb j_l\end{pmatrix}e_{j_1} \otimes \dotsb e_{j_l}.
\end{equation*}
Here the $\delta$ symbol appearing on the right hand side is 1 when the indices ``fit'', i.e. if each block of $\pi$ contains equal indices, and 0 otherwise.

It follows from the above discussion that $\Hom(u^{\otimes k}, u^{\otimes l})$ consists of certain linear combinations of the operators $T_\pi$, with $\pi \in \mc P(k,l)$.  We call $G$ ``easy'' if these spaces are spanned by partitions.

\begin{definition}[\cite{bsp}]\label{easydef}
A compact orthogonal quantum group $S_n \subset G \subset O_n^+$ is called \textit{easy} if for each $k,l \in \N$, there exist sets $D(k,l) \subset \mc P(k,l)$ such that $\Hom(u^{\otimes k}, u^{\otimes l}) = \mathrm{span}(T_\pi:\pi \in D(k,l))$.  If we have $D(k,l) \subset NC(k,l)$ for each $k,l \in \N$, we say that $G$ is a \textit{free quantum group}.
\end{definition}

There are four natural examples of classical groups which are easy: 
\begin{center}
 \begin{tabular}{|l|l|}
  \hline
Group & Partitions\\
\hline
Permutation group $S_n$ & $\mc P$: All partitions\\
\hline
Orthogonal group $O_n$ & $\mc P_2$: Pair partitions\\
\hline
Hyperoctahedral group $H_n$ & $\mc P_h$: Partitions with even block sizes\\
\hline
Bistochastic group $B_n$ & $\mc P_b$: Partitions with block size $\leq 2$\\
\hline
 \end{tabular}
\end{center}

The free quantum groups $O_n^+,S_n^+,H_n^+$ and $B_n^+$ are obtained by restricting to noncrossing partitions:
\begin{center}
 \begin{tabular}{|l|l|}
  \hline
Quantum group & Partitions\\
\hline
$S_n^+$ & $NC$: All noncrossing partitions\\
\hline
$O_n^+$ & $NC_2$: Noncrossing pair partitions\\
\hline
$H_n^+$ & $NC_h$: Noncrossing partitions with even block sizes\\
\hline
$B_n^+$ & $NC_b$: Noncrossing partitions with block size $\leq 2$\\
\hline
 \end{tabular}
\end{center}

For further discussion of easy quantum groups and their classification, see \cite{bsp},\cite{bcs1}.  

\subsection{Weingarten formula}

It is a fundamental result of Woronowicz \cite{wor1} that if $G$ is a compact orthogonal quantum group, then there is a unique \textit{Haar state} $\int:C(G) \to \C$ which is left and right invariant in the sense that
\begin{equation*}
 (\textstyle \int \otimes \mathrm{id})\Delta(f) = (\textstyle\int f) \cdot 1_{C(G)} = (\mathrm{id} \otimes \textstyle\int)\Delta(f), \qquad (f \in C(G)).
\end{equation*}

One very useful consequence of the ``easiness'' condition is that it leads to a combinatorial \textit{Weingarten formula} for computing the Haar state, which we recall from \cite{col},\cite{bc1},\cite{bc2},\cite{bsp}.

\begin{definition}
Let $D(k) \subset NC(k)$ be a collection of noncrossing partitions.  For $n \in \N$, define the \textit{Gram matrix} $(G_{D(k),n}(\pi,\sigma))_{\pi,\sigma \in D(k)}$ by
\begin{equation*}
 G_{D(k),n}(\pi,\sigma) = n^{|\pi \vee \sigma|}.
\end{equation*}
(Note that the join $\vee$ is taken in $\mc P(k)$, so that $\pi \vee \sigma$ may have crossings even if $\pi$ and $\sigma$ do not).  $G_{D(k),n}$ is invertible for $n \geq 4$, let $W_{D(k),n}$ denote its inverse.
\end{definition}

\begin{theorem}
Let $G \subset O_n^+$ be a free quantum group, and let $D(k) \subset NC(0,k)$ be the associated partitions with no upper points.  If $n \geq 4$, then for any $1 \leq i_1,j_1,\dotsc,i_k,j_k \leq n$ we have
\begin{equation*}
 \int_G u_{i_1j_1}\dotsb u_{i_kj_k} = \sum_{\substack{\pi,\sigma \in D(k)\\ \pi \leq \ker \mathbf i\\ \sigma \leq \ker \mathbf j}} W_{D(k),n}(\pi,\sigma).
\end{equation*}
 \qed
\end{theorem}

We will assume throughout the paper that $n \geq 4$, so that the Weingarten formula above is valid.  This reduces the problem of evaluating integrals over a free quantum group $G$ to computing the entries of the corresponding Weingarten matrix.  We recall the following result from \cite{bcs2}, which allows us to control the asymptotic behavior of $W_{D(k),n}$ as $n \to \infty$.

\begin{theorem}\label{West}
Let $G$ be a free quantum group with partitions $D(k) \subset NC(k)$.  Then for any $\pi,\sigma \in D(k)$ we have
\begin{equation*}
 n^{|\pi|}W_{D(k),n}(\pi,\sigma) = \mu_k(\pi,\sigma) + O(n^{-1})
\end{equation*}
as $n \to \infty$, where $\mu_k$ is the M\"{o}bius function on $NC(k)$. \qed
\end{theorem}

\section{Operator-valued $R$-cyclic families}\label{sec:rcyclic}

In this section we develop some of the basic theory of operator-valued $R$-cylic families of matrices.  This generalizes some results from \cite{nss} in the scalar case.  Throughout this section, $(A,E:\mc A \to \mc B)$ will be a fixed operator-valued probability space.
 
\begin{definition}
Let $X_1,\dotsc,X_s$ be a family of matrices in $M_n(\mc A)$, $X_r = (x^{(r)}_{ij})_{1 \leq i,j \leq n}$.  We say that the family $X_1,\dotsc,X_s$ is a \textit{$\mc B$-valued $R$-cyclic family}, or \textit{$R$-cyclic with respect to $E$}, if for any $b_1,\dotsc,b_k \in \mc B$, $1 \leq r_1,\dotsc,r_k \leq s$ and $1 \leq i_1,j_1,\dotsc,i_k,j_k \leq n$ we have
\begin{equation*}
 \kappa_E^{(k)}[x^{(r_1)}_{i_kj_1}b_1,x^{(r_2)}_{i_1,j_2}b_2,\dotsc,x^{(r_k)}_{i_{k-1}j_k}b_k] = 0
\end{equation*}
unless $i_l = j_l$ for $1 \leq l \leq k$.  Equivalently, for $\sigma \in NC(k)$, $b_1,\dotsc,b_k \in \mc B$, $1 \leq r_1,\dotsc,r_k \leq s$ and $1 \leq i_{11},i_{12},\dotsc,i_{k2} \leq n$ we have
\begin{equation*}
 \kappa_E^{(\sigma)}[x^{(r_1)}_{i_{11}i_{12}}b_1,\dotsc,x^{(r_k)}_{i_{k1}i_{k2}}b_k] = 0
\end{equation*}
unless $\widetilde \sigma \leq \ker \mathbf i$, where we set $\ker \mathbf i = \ker (i_{11},i_{12},\dotsc,i_{k1},i_{k2}) \in \mc P(2k)$.
\end{definition}

Note that if $X_1,\dotsc,X_s$ is an $R$-cyclic family with respect to $E$, $X_r = (x_{ij}^{(r)})_{1 \leq i,j \leq n}$, then the $\mc B$-valued joint distribution of $(x^{(r)}_{ij})$ is determined by the ``cyclic'' cumulants
\begin{equation*}
 \kappa_E^{(k)}[x_{i_ki_1}^{(r_1)}b_1,x_{i_1i_2}^{(r_2)}b_2,\dotsc,x_{i_{k-1}i_k}^{(r_k)}b_k].
\end{equation*}
This is encoded in the \textit{$\mc B$-valued determining series} of the family $X_1,\dotsc,X_s$, which is defined to be the $\mc B$-linear map $\theta_X:\mc B \langle t^{(r)}_1,\dotsc,t^{(r)}_n: 1\leq r \leq s \rangle \to \mc B$ determined by
\begin{equation*}
 \theta_X(t_{i_1}^{(r_1)}b_1t^{(r_2)}_{i_2}b_2\dotsb t^{(r_k)}_{i_k}b_k) = \kappa_E^{(k)}[x_{i_ki_1}^{(r_1)}b_1,x^{(r_2)}_{i_1i_2}b_2,\dotsc,x_{i_{k-1}i_k}^{(r_k)}b_k].
\end{equation*}
(The terminology comes from case $\mc B = \C$, where $\theta_X$ can be expressed as a formal power series in the variables $t_i^{(r)}$, see \cite{nss},\cite{ns}).

\begin{remark}\label{rcycremark}
While $R$-cyclicity is defined in terms of the distributions of the entries of the matrices $X_1,\dotsc,X_s$, it turns out to be equivalent to a natural condition on the $\mc B$-valued distribution of $X_1,\dotsc,X_s$ in $M_n(\mc A)$.  Indeed, letting $\mc D$ denote the algebra of diagonal matrices in $M_n(\mc A)$ with entries from $\mc B$, we will show below that $X_1,\dotsc,X_s$ form an $R$-cyclic family if and only if they are free from $M_n(\mc B)$ with amalgamation over $\mc D$.

First we need to show that $R$-cyclicity is a property of the algebra $\mc C$ which is generated by $\{X_1,\dotsc,X_s\} \cup \mc D$.  In other words, $R$-cyclicity should be preserved by certain algebraic operations.  Clearly $R$-cyclicity of the family $X_1,\dotsc,X_s$ is preserved under reordering the matrices or deleting one.  Moreover,
\begin{enumerate}
 \item If $X_1,\dotsc,X_s$ are $R$-cyclic with respect to $E$, and $X$ is in the $\mc B - \mc B$ bimodule span of $X_1,\dotsc,X_s$, then $X_1,\dotsc,X_s,X$ is still $R$-cyclic.  This follows from the $\mc B-\mc B$ multilinearity of the cumulants $\kappa_E^{(k)}:\mc A^{k} \to \mc B$.
 \item If $X_1,\dotsc,X_s$ are $R$-cyclic with respect to $E$, and $D \in \mc D$ is a diagonal matrix with entries from $\mc B$, then $X_1,\dotsc,X_s,D$ is still $R$-cyclic.  This is due to the fact that a $\mc B$-valued cumulant $\kappa_E^{(k)}$ with $k \geq 2$ is zero if any of its entries are from $\mc B$.
\end{enumerate}
 
We will now show that $R$-cyclicity is also preserved under taking products.  Note that from (2) above we may first add the identity matrix to the family $X_1,\dotsc,X_s$, so that the $R$-cyclic family constructed in the lemma still contains $X_1,\dotsc,X_s$.
\end{remark}

\begin{lemma}\label{rcycproducts}
Let $(X_1,\dotsc,X_s)$ be a $\mc B$-valued $R$-cyclic family in $M_n(\mc A)$.  Then the family $(X_{r_1}\cdot X_{r_2})_{1 \leq r_1,r_2 \leq s}$ is $R$-cyclic with respect to $E$.
\end{lemma}

\begin{proof}
Fix $1 \leq r_{11},r_{12},\dotsc,r_{k1},r_{k2} \leq s$ and $b_1,\dotsc,b_k \in \mc B$, we must show that
\begin{multline*}
 \kappa_{E}^{(k)}[(X_{r_{11}}X_{r_{12}})_{i_{k}j_1}b_1,\dotsc,(X_{r_{k1}}X_{r_{k2}})_{i_{k-1}j_k}b_k]\\
= \sum_{1 \leq l_1,\dotsc,l_k \leq n} \kappa_E^{(k)}[x^{(r_{11})}_{i_{k}l_1}x^{(r_{12})}_{l_1j_1}b_1,\dotsc,x^{(r_{k1})}_{i_{k-1}l_k}x^{(r_{k2})}_{l_kj_{k}}b_k]
\end{multline*}
is equal to zero unless $i_1 = j_1,\dotsc,i_k = j_k$.  In fact for fixed $1 \leq l_1,\dotsc,l_k \leq n$ the term appearing above is zero unless this condition holds.  Indeed, using the formula for cumulants of products from \cite{sp2} we have
\begin{equation*}
\kappa_E^{(k)}[x^{(r_{11})}_{i_kl_1}x^{(r_{12})}_{l_1j_1}b_1,\dotsc,x^{(r_{k1})}_{i_{k-1}l_k}x^{(r_{k2})}_{l_kj_k}b_k] = \sum_{\substack{\sigma \in NC(2k)\\ \sigma \vee \hat 0_k = 1_{2k}}} \kappa_E^{(\sigma)}[x^{(r_{11})}_{i_{k}l_1},x^{(r_{12})}_{l_1j_1}b_1,\dotsc,x^{(r_{k1})}_{i_{k-1}l_k},x^{(r_{k2})}_{l_kj_k}b_k]
\end{equation*}
Now from the $R$-cyclicity condition we have 
\begin{equation*}
 \kappa_E^{(\sigma)}[x^{(r_{11})}_{i_{k}l_1},x^{(r_{12})}_{l_1j_{1}}b_1,\dotsc,x^{(r_{k1})}_{i_{k-1}l_k},x^{(r_{k2})}_{l_kj_{k}}b_k] = 0
\end{equation*}
unless $\widetilde \sigma \leq \ker (i_{k},l_1,l_1,j_{1},\dotsc,l_k,j_{k})$.  From Lemma \ref{fatfacts}, this is equivalent to
\begin{equation*}
 \widetilde{K(\sigma)} \leq \ker (l_1,l_1,j_{1},i_1,\dotsc,l_k,l_k,j_k,i_k).
\end{equation*}
Let $\tau$ be the partition $\{\{1,2\},\{3\},\{4\},\{5,6\},\{7\},\{8\},\dotsc,\{4k-3,4k-2\},\{4k-1\},\{4k\}\}$, we claim that $\widetilde{K(\sigma)} \vee \tau = \widehat{K(\sigma)}$.  The result will then follow, as 
\begin{equation*}
  \widetilde{K(\sigma)} \leq \ker (l_1,l_1,j_{1},i_1,\dotsc,l_k,l_k,j_k,i_k) \Leftrightarrow  \widetilde{K(\sigma)} \vee \tau \leq \ker (l_1,l_1,j_{1},i_1,\dotsc,l_k,l_k,j_k,i_k),
\end{equation*}
and if $\widehat{K(\sigma)} \leq \ker (l_1,l_1,j_{1},i_1,\dotsc,l_k,l_k,j_k,i_k)$ then we must have $i_1 = j_1,\dotsc,i_k = j_k$.

To prove the claim, first note that the join of any $\sigma \in NC(2k)$ with $\hat 0_k$ is noncrossing (as $\hat 0_k$ is an interval partition).  So we may apply the Kreweras complement to both sides of the equation $\sigma \vee \hat 0_k = 1_{2k}$ to see
\begin{equation*}
 \sigma \vee \hat 0_k = 1_{2k} \Leftrightarrow K(\sigma) \wedge 0_k \wr 1_k = 0_{2k}.
\end{equation*}
So if $\sigma \vee \hat 0_k = 1_{2k}$, then no block of $K(\sigma)$ may contain more than one even number.  Let $V = (l_1 < \dotsc < l_m)$ be a block of $K(\sigma)$, so that $\widetilde{K(\sigma)}$ has pairings $(2l_1-1,2l_m),(2l_1,2l_2-1),\dotsc,(2l_{m-1},2l_m-1)$.  Since $V$ contains at most one even number, $\tau$ contains all of the pairs $\{(2l_p-1,2l_p): 1 \leq p \leq m\}$, except for at most one.  But it is then clear that $\widetilde{K(\sigma)} \vee \tau$ contains the block $\{2l_1-1,2l_1,\dotsc,2l_m-1,2l_m\}$, so that $\widetilde{K(\sigma)} \vee \tau = \widehat{K(\sigma)}$ as claimed.
\end{proof}

\begin{proposition}\label{rcycalg}
Let $X_1,\dotsc,X_s$ be a $\mc B$-valued $R$-cyclic family in $M_n(\mc A)$.  Let $\mc D$ denote the algebra of diagonal matrices with entries in $\mc B$, and let $\mc C$ denote the subalgebra of $M_n(\mc A)$ which is generated by $\{X_1,\dotsc,X_s\} \cup \mc D$.  Then any finite family of matrices from $\mc C$ is $R$-cyclic with respect to $E$.
\end{proposition}

\begin{proof}
This follows from combining Remark \ref{rcycremark} with Lemma \ref{rcycproducts}.
\end{proof}

Let $V_{ij}$ denote the natural matrix units in $M_n(\mc A)$, i.e. $V_{ij}$ has $(i,j)$-entry $1$ and all other entries $0$.  There are natural conditional expectations $E_{M_n(\mc B)}:M_n(\mc A) \to M_n(\mc B)$, $E_{\mc D}:M_n(A) \to \mc D$ and $E_{\mc B}:M_n(\mc A) \to \mc B$, given by the formulas
\begin{align*}
 E_{M_n(\mc B)}\bigl[(a_{ij})_{1 \leq i,j \leq n}\bigr] &= \sum_{1 \leq i,j \leq n} E[a_{ij}] \cdot V_{ij}\\
E_{\mc D}\bigl[(a_{ij})_{1 \leq i,j \leq n}\bigr] &= \sum_{i=1}^n E[a_{ii}] \cdot V_{ii}\\
E_{\mc B}\bigl[(a_{ij})_{1 \leq i,j \leq n}\bigr] &= E\bigl[\tr\bigl((a_{ij})\bigr)\bigr] = n^{-1}\sum_{i=1}^n E[a_{ii}].
\end{align*}
Note that $E_{\mc D} \circ E_{M_n(\mc B)} = E_{\mc D}$ and $E_{\mc B} \circ E_{\mc D} = E_{\mc B}$.  The following lemma connects the $\mc D$-valued distribution of a $\mc B$-valued $R$-cyclic family $X_1,\dotsc,X_s$ with the ``cyclic'' cumulants of their entries with respect to $E$.

\begin{lemma}\label{cycliccumulants}
Let $X_1,\dotsc,X_s$ be a $\mc B$-valued $R$-cyclic family in $M_n(\mc A)$.  Then for any $1 \leq i_1,\dotsc,i_{k-1} \leq n$ and $b_1,\dotsc,b_k \in \mc B$, we have
\begin{equation*}
 \kappa_{E_{\mc D}}^{(k)}[X_{r_1}b_1V_{i_1i_1},\dotsc,X_{r_{k-1}}b_{k-1}V_{i_{k-1}i_{k-1}},X_{r_k}b_k] = \sum_{1 \leq i_k \leq n} \kappa_{E}^{(k)}[x^{(r_1)}_{i_ki_1}b_1,\dotsc,x^{(r_k)}_{i_{k-1}i_k}b_k] \cdot V_{i_ki_k}.
\end{equation*}

\end{lemma}

\begin{proof}
We claim that
\begin{equation*}
 E_{\mc D}^{(\sigma)}[X_{r_1}b_1V_{i_1i_1},\dotsc,X_{r_k}b_kV_{i_ki_k}] = E^{(\sigma)}[x^{(r_1)}_{i_ki_1}b_1,\dotsc,x_{i_{k-1}i_k}b_k]\cdot V_{i_ki_k}
\end{equation*}
for any $\sigma \in NC(k)$, from which the result follows by M\"{o}bius inversion.

We prove this by induction on the number of blocks of $\sigma$, the case $\sigma = 1_k$ is trivial.  So let $V = \{l+1,\dotsc,l+s\}$ be an interval of $\sigma$, then 
\begin{multline*}
 E_{\mc D}^{(\sigma)}[X_{r_1}b_1V_{i_1i_1},\dotsc,X_{r_k}b_kV_{i_ki_k}] \\
= E_{\mc D}^{(\sigma \setminus V)}[X_{r_1}b_1V_{i_1i_1},\dotsc,X_{r_l}b_lV_{i_li_l}E_{\mc D}[X_{r_{l+1}}b_{l+1}V_{i_{l+1}i_{l+1}}\dotsb X_{r_{l+s}}b_{l+s}V_{i_{l+s}i_{l+s}}],\dotsc, X_{r_k}b_kV_{i_ki_k}]\\
= \delta_{i_li_{l+s}}E_{\mc D}^{(\sigma \setminus V)}[X_{r_1}b_1V_{i_1i_1},\dotsc,X_{r_l}b_lE[x^{(r_{l+1})}_{i_{l}i_{l+1}}b_{l+1}\dotsb x^{(r_{l+s})}_{i_{l+s-1}i_{l+s}}]V_{i_li_l},\dotsc,X_{r_k}b_kV_{i_ki_k}]\\
= \delta_{i_li_{l+s}}E^{(\sigma)}[x^{(r_1)}_{i_ki_1}b_1,\dotsc,x_{i_{k-1}i_k}b_k]\cdot V_{i_ki_k},
\end{multline*}
where we have used the induction hypothesis on the last line.  

So it remains only to see that
\begin{equation*}
E^{(\sigma)}[x^{(r_1)}_{i_ki_1}b_1,\dotsc,x^{(r_k)}_{i_{k-1}i_k}b_k] = 0
\end{equation*}
if $i_l \neq i_{l+s}$.  We have
\begin{equation*}
 E^{(\sigma)}[x^{(r_1)}_{i_ki_1}b_1,\dotsc,x^{(r_k)}_{i_{k-1}i_k}b_k] = \sum_{\substack{\pi \in NC(k)\\ \pi \leq \sigma}} \kappa_E^{(\pi)}[x^{(r_1)}_{i_ki_1}b_1,\dotsc,x^{(r_k)}_{i_{k-1}i_k}b_k],
\end{equation*}
so we claim that if $\pi \leq \sigma$ then $\kappa_E^{(\pi)}[x^{(r_1)}_{i_ki_1}b_1,\dotsc,x^{(r_k)}_{i_{k-1}i_k}b_k] = 0$ unless $i_l = i_{l+s}$.  Since $X_1,\dotsc,X_s$ are $R$-cyclic with respect to $E$, we have $\kappa_E^{(\pi)}[x^{(r_1)}_{i_ki_1}b_1,\dotsc,x^{(r_k)}_{i_{k-1}i_k}b_k] = 0$ unless $\widetilde \pi \leq \ker (i_k,i_1,i_1,\dotsc,i_{k-1}i_k)$.  From Lemma \ref{fatfacts},
\begin{equation*}
\widetilde \pi \leq \ker (i_k,i_1,i_1,\dotsc,i_{k-1},i_k) \Leftrightarrow K(\pi) \leq \ker \mathbf i.
\end{equation*}
Since $\pi \leq \sigma$, we have $K(\sigma) \leq K(\pi)$.  In particular, $l$ and $l+s$ are in the same block of $K(\pi)$, and the result follows.

\end{proof}

We are now prepared to prove the main result of this section.

\begin{theorem}\label{rcyclictheorem}
Let $X_1,\dotsc,X_s$ be a family of matrices in $M_n(\mc A)$, and let $\mc C$ denote the algebra generated by $\{X_1,\dotsc,X_s\} \cup \mc D$.  Then $X_1,\dotsc,X_s$ is $R$-cyclic with respect to $E$ if and only if $\mc C$ is free from $M_n(\mc B)$ with amalgamation over $\mc D$.
\end{theorem}

\begin{proof}
Suppose that $X_1,\dotsc,X_s$ form an $R$-cyclic family with respect to $E$.  Let $Y_1,\dotsc,Y_k \in \mc C$ and $B_1,\dotsc,B_k \in M_n(\mc B)$
\begin{align*}
Y_l &= (y^{(l)}_{ij})_{1 \leq i,j \leq n} \qquad (1 \leq l \leq k)\\
B_l &= (b^{(l)}_{ij})_{1 \leq i,j \leq n} \qquad (1 \leq l \leq k).
\end{align*}
Assume that $E_{\mc D}[Y_l] = 0$ for $2 \leq l \leq k$, $E_{\mc D}[B_l] = 0$ for $1 \leq l \leq k -1$, and that at most one of $E_{\mc D}[Y_1]$ and $E_{\mc D}[B_k]$ is nonzero.  We need to show that
\begin{equation*}
 E_{\mc D}[Y_1B_1\dotsb Y_kB_k] = 0.
\end{equation*}

We have
\begin{align*}
 E_{\mc D}[Y_1B_1\dotsb Y_kB_k] &= \sum_{1 \leq i_1,\dotsc,i_{2k} \leq n} E[y^{(1)}_{i_1i_2}b^{(1)}_{i_2i_3}\dotsb y^{(k)}_{i_{2k-1}i_{2k}}b^{(k)}_{i_{2k}i_1}]\cdot V_{i_1i_1}\\
&= \sum_{1 \leq i_1,\dotsc,i_{2k} \leq n} \sum_{\sigma \in NC(k)} \kappa_E^{(\sigma)}[y^{(1)}_{i_1i_2}b^{(1)}_{i_2i_3},\dotsc,y^{(k)}_{i_{2k-1}i_{2k}}b^{(k)}_{i_{2k}i_1}]\cdot V_{i_1i_1}.
\end{align*}

Now $Y_1,\dotsc,Y_k$ are $R$-cyclic by Proposition \ref{rcycalg}, so we have
\begin{equation*}
\kappa_E^{(\sigma)}[y^{(1)}_{i_1i_2}b^{(1)}_{i_2i_3},\dotsc,y^{(k)}_{i_{2k-1}i_{2k}}b^{(k)}_{i_{2k}i_1}] = 0
\end{equation*}
unless $\widetilde \sigma \leq \ker \mathbf i$.  Suppose that $\sigma$ has an interval $V = \{l,\dotsc,l+m\}$ with $m \geq 1$.  Then $\widetilde \sigma$ contains the pair $(2l,2l+1)$, so $\widetilde \sigma \leq \ker \mathbf i$ forces $i_{2l} = i_{2l+1}$.  But $E_{\mc D}[B_l] = 0$ implies $b^{(l)}_{i_{2l}i_{2l}} = 0$, and so we have
\begin{equation*}
\kappa_E^{(\sigma)}[y^{(1)}_{i_1i_2}b^{(1)}_{i_2i_3},\dotsc,y^{(l)}_{i_{2l-1}i_{2l}}b^{(l)}_{i_{2l}i_{2l}},\dotsc ,y^{(k)}_{i_{2k-1}i_{2k}}b^{(k)}_{i_{2k}i_1}] = 0.
\end{equation*}

We are now left to consider $\sigma$ which contain no such interval.  If $k > 1$, then it follows that $\sigma$ must have a singleton $\{l\}$, $l > 1$.  But now $\sigma \leq \ker \mathbf i$ forces $i_{2l-1} = i_{2l}$, so that we have
\begin{equation*}
 E[y^{(l)}_{i_{2l-1}i_{2l}}b^{(l)}_{i_{2l}i_{2l+1}}] = E[y^{(l)}_{i_{2l}i_{2l}}]\cdot b^{(l)}_{i_{2l}i_{2l+1}} = 0,
\end{equation*}
since $E_{\mc D}[Y_l] = 0$.  It follows that
\begin{multline*}
 \kappa_E^{(\sigma)}[y^{(1)}_{i_1i_2}b^{(1)}_{i_2i_3},\dotsc,y^{(k)}_{i_{2k-1}i_{2k}}b^{(k)}_{i_{2k}i_1}]\\
= \kappa_{E}^{(\sigma \setminus \{l\})}[y^{(1)}_{i_1i_2}b^{(1)}_{i_2i_3},\dotsc,y^{(l-1)}_{i_{2l-3}i_{2l-2}}b^{(l-1)}_{i_{2l-2}i_{2l-1}} E[y^{(l)}_{i_{2l-1}i_{2l}}b^{(l)}_{i_{2l}i_{2l+1}}],\dotsc,y^{(k)}_{i_{2k-1}i_{2k}}b^{(k)}_{i_{2k}i_{1}}]
\end{multline*}
is equal to $0$.

Finally, if $k = 1$ then we are considering
\begin{equation*}
 E[y^{(1)}_{i_1i_1}]\cdot b^{(1)}_{i_1i_1} = 0,
\end{equation*}
since either $E_{\mc D}[Y_1] = 0$ or $E_{\mc D}[B_1] = 0$.  So we have proved that $\mc C$ is free from $M_n(\mc B)$ with amalgamation over $\mc D$.

Now suppose that $\mc C$ is free from $M_n(\mc B)$ with amalgamation over $\mc D$, we will show that $X_1,\dotsc,X_s$ are $R$-cyclic with respect to $E$.  Let $\{y_{ij}^{(r)}: 1 \leq i,j \leq n, 1 \leq r \leq s\}$ be random variables in a different $\mc B$-valued probability space $(\mc A', E':\mc A' \to \mc B)$ such that
\begin{equation*}
 \kappa_{E'}^{(k)}[y_{i_kj_1}^{(r_1)}b_1,\dotsc,y_{i_{k-1}j_k}^{(r_k)}b_k] = \begin{cases}
                                                                              \bigl(\kappa_{E_{\mc D}}[X_{r_1}b_1V_{i_1i_1},\dotsc,X_{r_{k}}b_kV_{i_ki_k}]\bigr)_{i_ki_k}, & i_1 = j_1,\dotsc,i_k = j_k\\
0, & \text{otherwise}
                                                                             \end{cases}.
\end{equation*}
Such a construction is always possible, see e.g. \cite{sp1}.  

For $1 \leq r \leq s$, let $Y_r = (y_{ij}^{(r)})_{1 \leq i,j \leq n} \in M_n(\mc A')$.  From Lemma \ref{cycliccumulants} we have
\begin{align*}
 \kappa_{E'_{\mc D}}^{(k)}[Y_{r_1}b_1V_{i_1i_1},\dotsc,Y_{r_{k-1}}b_{k-1}V_{i_{k-1}i_{k-1}},Y_{r_k}b_k] &= \sum_{i_k=1}^n \kappa_{E'}^{(k)}[y^{(r_1)}_{i_ki_1}b_1,\dotsc,y^{(r_k)}_{i_{k-1}i_k}b_k]\cdot V_{i_ki_k}\\
&= \kappa_{E_{\mc D}}^{(k)}[X_{r_1}b_1V_{i_1i_1},\dotsc,X_{r_{k-1}}b_{k-1}V_{i_{k-1}i_{k-1}},X_{r_k}b_k].
\end{align*}
Since $\mc D$ is spanned by elements of the form $b\cdot V_{ii}$ for $b \in \mc B$ and $1 \leq i \leq n$, it follows that $X_1,\dotsc,X_s$ and $Y_1,\dotsc,Y_s$ have the same $\mc D$-valued distribution.  

Now the family $Y_1,\dotsc,Y_s$ is $R$-cyclic with respect to $E'$ by construction, and therefore by the implication proved above, the algebra $\mc C' \subset M_n(\mc A')$ generated by $\{Y_1,\dotsc,Y_s\} \cup \mc D$ is free from $M_n(\mc B)$, with amalgamation over $\mc D$.  But this means that the distribution of the family $Y_1,\dotsc,Y_s$ with respect to $M_n(\mc B)$ is determined by its distribution with respect to $\mc D$ (see e.g. \cite{nss2}).  Likewise, since $X_1,\dotsc,X_s$ are free from $M_n(\mc B)$ with amalgamation over $\mc D$, the distribution of $X_1,\dotsc,X_s$ with respect to $M_n(\mc B)$ is determined by its distribution with respect to $\mc D$.  So since $Y_1,\dotsc,Y_s$ and $X_1,\dotsc,X_s$ have the same $\mc D$-valued distribution, they also have the same $M_n(\mc B)$-valued distribution.

But now we have
\begin{align*}
E[x^{(r_1)}_{i_1j_1}b_1\dotsb x^{(r_k)}_{i_kj_k}b_k] \cdot V_{11} &= E_{M_n(\mc B)}[V_{1i_1}X_{r_1}b_1V_{j_11}\dotsb V_{1i_k}X_{r_k}b_kV_{j_k1}]\\
&=E'_{M_n(\mc B)}[V_{1i_1}Y_{r_1}b_1V_{j_11}\dotsb V_{1i_k}Y_{r_k}b_kV_{j_k1}]\\
&= E'[y^{(r_1)}_{i_1j_1}b_1\dotsb y^{(r_k)}_{i_kj_k}b_k] \cdot V_{11}.
\end{align*}
So $(x^{(r)}_{ij})$ and $(y^{(r)}_{ij})$ have the same $\mc B$-valued distribution, and since $Y_1,\dotsc,Y_s$ are $R$-cyclic with respect to $E'$, it follows that $X_1,\dotsc,X_s$ are $R$-cyclic with respect to $E$.
\end{proof}

\subsection{Uniform $R$-cyclicity}  We have shown that $R$-cyclic families of matrices are characterized by being free from $M_n(\mc B)$, with amalgamation over $\mc D$.  Since we have $\mc B \subset \mc D \subset M_n(\mc B)$, freeness from $M_n(\mc B)$ with amalgamation over $\mc B$ is a stronger condition than freeness with amalgamation over $\mc D$.  We will now show that this stronger condition can also be characterized by a stronger $R$-cyclicity condition.

\begin{definition}
Let $X_1,\dotsc,X_s$ be a family of matrices in $M_n(\mc A)$, $X_r = (x_{ij}^{(r)})_{1 \leq i,j \leq n}$.  We say that the family $X_1,\dotsc,X_s$ is \textit{uniformly $R$-cyclic with respect to $E$} if
\begin{equation*}
 \kappa_{E}^{(k)}[x^{(r_1)}_{i_kj_1}b_1,x_{i_1j_2}b_2,\dotsc,x^{(r_k)}_{i_{k-1}j_k}b_k] = \begin{cases} \kappa_E^{(k)}[x^{(r_1)}_{11}b_1,\dotsc,x^{(r_k)}_{11}b_k], & i_1 = j_1,\dotsc,i_k = j_k\\
 0, & \text{otherwise}
\end{cases}.
\end{equation*}

\end{definition}

We will characterize uniformly $R$-cyclic families by using Theorem \ref{rcyclictheorem} and a formulation of freeness in terms of factorization of cumulants from \cite{nss2}.  In our context their result is as follows:

\begin{proposition}\cite[Theorem 3.5]{nss2}\label{cumulantfactor}
Let $X_1,\dotsc,X_s \in M_n(\mc A)$, then $\{X_1,\dotsc,X_s\}$ is free from $\mc D$ with amalgamation over $\mc B$ if and only if
\begin{align*}
\kappa_{E_{\mc D}}^{(k)}[X_{r_1}b_1V_{i_1i_1},\dotsc,X_{r_{k-1}}b_{k-1}&V_{i_{k-1}i_{k-1}},X_{r_k}b_k] \\
&= \tr\bigl(\kappa_{E_\mc D}^{(k)}[X_{r_1}\tr(b_1V_{i_1i_1}),\dotsc,X_{r_{k-1}}\tr(b_{k-1}V_{i_{k-1}i_{k-1}}),X_{r_k}b_k]\bigr)\\
&= n^{1-k}\tr(\kappa_{E_{\mc D}}^{(k)}[X_{r_1}b_1,\dotsc,X_{r_k}b_k])
\end{align*}
for any $b_1,\dotsc,b_k \in \mc B$, $1 \leq r_1,\dotsc,r_k \leq s$ and $1 \leq i_1,\dotsc,i_{k-1} \leq n$.  Equivalently,
\begin{equation*}
\kappa_{E_\mc D}^{(k)}[X_{r_1}b_1V_{i_1i_1},\dotsc,X_{r_{k-1}}b_{k-1}V_{i_{k-1}i_{k-1}},X_{r_k}b_k] = n^{1-k}\kappa_{E_{\mc B}}^{(k)}[X_{r_1}b_1,\dotsc,X_{r_k}b_k].
\end{equation*}
 \qed
\end{proposition}

We now show that freeness from $M_n(\mc B)$ with amalgamation over $\mc B$ is characterized by uniform $R$-cyclicity, which establishes the equivalence of (2) and (3) in Theorem \ref{orthogonalcase}.

\begin{theorem}\label{uniformrcyclictheorem}
Let $X_1,\dotsc,X_s$ be a family of matrices in $M_n(\mc A)$, $X_r = (x^{(r)}_{ij})_{1 \leq i,j \leq n}$.  Then the family $X_1,\dotsc,X_s$ is uniformly $R$-cyclic with respect to $E$ if and only if $\{X_1,\dotsc,X_s\}$ is free from $M_n(\mc B)$ with amalgamation over $\mc B$.
\end{theorem}

\begin{proof}
First we claim that if $X_1,\dotsc,X_s$ is $R$-cyclic with respect to $E$, then it is uniformly $R$-cyclic if and only if $\{X_1,\dotsc,X_s\}$ is free from $\mc D$ with amalgamation over $\mc B$.  Indeed, from Lemma \ref{cycliccumulants} we have
\begin{equation*}
\kappa_{E_\mc D}^{(k)}[X_{r_1}b_1V_{i_1i_1},\dotsc,X_{r_{k-1}}b_{k-1}V_{i_{k-1}i_{k-1}},X_{r_k}b_k] = \sum_{i_k=1}^n \kappa_E^{(k)}[x^{(r_1)}_{i_ki_1}b_1,\dotsc,x^{(r_k)}_{i_{k-1}i_k}b_k] \cdot V_{i_{k}i_k}.
\end{equation*}
Now if $\{X_1,\dotsc,X_s\}$ is free from $\mc D$ with amalgamation over $\mc B$, then combining this equation with Proposition \ref{cumulantfactor} we have
\begin{equation*}
 \kappa_E^{(k)}[x^{(r_1)}_{i_ki_1}b_1,\dotsc,x^{(r_k)}_{i_{k-1}i_k}b_k] = n^{1-k}\kappa_{E_{\mc B}}^{(k)}[X_{r_1}b_1,\dotsc,X_{r_k}b_k].
\end{equation*}
Since the right hand side does not depend on the indices $i_1,\dotsc,i_k$, $X$ is uniformly $R$-cyclic as claimed.  Conversely if $X$ is uniformly $R$-cyclic then
\begin{align*}
\kappa_{E_\mc D}^{(k)}[X_{r_1}b_1V_{i_1i_1},\dotsc,b_{k-1}&V_{i_{k-1}i_{k-1}},X_{r_k}b_k]\\
&= \sum_{j_k=1}^n\kappa_{E}^{(k)}[x^{(r_1)}_{j_ki_1}b_1,\dotsc,x^{(r_k)}_{i_{k-1}j_k}b_k]\cdot V_{j_kj_k}\\
&= \kappa_{E}^{(k)}[x^{(r_1)}_{11}b_1,\dotsc,x^{(r_k)}_{11}b_k]\\
&= n^{1-k} \sum_{1 \leq j_1,\dotsc,j_{k-1} \leq n} \kappa_{E}^{(k)}[x^{(r_1)}_{1j_1}b_1,\dotsc,x^{(r_k)}_{j_{k-1}1}b_k]\\
&= n^{1-k}\sum_{1 \leq j_1,\dotsc,j_{k-1} \leq n}\tr\bigl(\kappa_{E_\mc D}^{(k)}[X_{r_1}b_1V_{j_1j_1},\dotsc,X_{r_{k-1}}b_{k-1}V_{j_{k-1}j_{k-1}},X_{r_k}b_k])\\
&= n^{1-k}\tr\bigl(\kappa_{E_{\mc D}}^{(k)}[X_{r_1}b_1,\dotsc,X_{r_k}b_k]\bigr).
\end{align*}
The claim then follows from Proposition \ref{cumulantfactor}.

Now suppose that $X_1,\dotsc,X_s$ is uniformly $R$-cyclic with respect to $E$.  By Theorem \ref{rcyclictheorem}, $\{X_1,\dotsc,X_s\}$ is free from $M_n(\mc B)$ with amalgamation over $\mc D$.  By the above claim, $\{X_1,\dotsc,X_s\}$ is free from $\mc D$ with amalgamation over $\mc B$.  It then follows from \cite[Proposition 3.7]{nss2} that $\{X_1,\dotsc,X_s\}$ is free from $M_n(\mc B)$ with amalgamation over $\mc B$.

Conversely, if $\{X_1,\dotsc,X_s\}$ is free from $M_n(\mc B)$ with amalgamation over $\mc B$, then it is also free from $M_n(\mc B)$ with amalgamation over $\mc D$ and hence $R$-cyclic by Theorem \ref{rcyclictheorem}.  But since $\{X_1,\dotsc,X_s\}$ is free from $\mc D$ with amalgamation over $\mc B$, it follows from the above claim that $X_1,\dotsc,X_s$ is uniformly $R$-cyclic with respect to $E$.
\end{proof}

We will now prove Theorem \ref{infdivis}, which relates uniform $R$-cyclicity and free infinite divisibility.  First we prove a version of that theorem which holds for finite matrices.

\begin{theorem}\label{finitedivis}
Let $\mc A$ be a unital C$^*$-algebra, $1 \in \mc B \subset \mc A$ a C$^*$-subalgebra and $E:\mc A \to \mc B$ a faithful completely positive conditional expectation.  Let $y_1,\dotsc,y_s \in \mc A$, then the following conditions are equivalent:
\begin{enumerate}
 \item There is a unital C$^*$-algebra $\mc A'$, a unital inclusion $\mc B \hookrightarrow \mc A'$, a faithful completely positive conditional expectation $E':\mc A' \to \mc B$, and a family $\{x_{ij}^{(r)}:1 \leq i,j \leq n, 1 \leq r \leq s\} \subset \mc A'$ such that
\begin{itemize}
\item $(x_{11}^{(1)},\dotsc,x_{11}^{(s)})$ has the same $\mc B$-valued distribution as $(y_1,\dotsc,y_s)$.
\item $X_1,\dotsc,X_s$ forms a $\mc B$-valued uniformly $R$-cyclic family, where $X_r = (x_{ij}^{(r)})_{1 \leq i,j \leq n}$ for $1 \leq r \leq s$.
\end{itemize}
\item The $\mc B$-valued joint distribution of $(y_1,\dotsc,y_s)$ is $n$-times freely divisible, i.e. there exists a unital C$^*$-algebra $\mc A_n$, a unital inclusion $\mc B \hookrightarrow \mc A_n$, a faithful completely positive conditional expectation $E_n:\mc A_n \to \mc B$, and a family $\{x_r^{(i)}: 1 \leq i \leq n, 1 \leq r \leq s\}$ such that 
\begin{itemize}
\item The families $\{y_1^{(1)},\dotsc,y_s^{(1)}\},\dotsc,\{y_{1}^{(n)},\dotsc,y_s^{(n)}\}$ are freely independent with respect to $E_n$.
\item The $\mc B$-valued distribution of $(y_1^{(i)},\dotsc,y_s^{(i)})$ does not depend on $1 \leq i \leq n$.
\item $(y_1,\dotsc,y_s)$ has the same $\mc B$-valued distribution as $(y'_1,\dotsc,y'_s)$, where $y'_r = y_r^{(1)} + \dotsb + y_r^{(n)}$ for $1 \leq r \leq s$.
\end{itemize}
\end{enumerate}
  
\end{theorem}

\begin{proof}
First assume that (1) holds.  Let $\{y_r^{(i)}: 1 \leq i \leq n, 1 \leq r \leq s\}$ be a family of random variables in a unital C$^*$-algebra $\mc A_n$ which contains $\mc B$ as a unital C$^*$-subalgebra, with a faithful completely positive conditional expectation $E_n:\mc A_n \to \mc B$ such that
\begin{itemize}
\item The families $\{y_1^{(1)},\dotsc,y_s^{(1)}\},\dotsc,\{y_{1}^{(n)},\dotsc,y_s^{(n)}\}$ are free with amalgamation over $\mc B$.
\item The $\mc B$-valued distribution of $(y_1^{(i)},\dotsc,y_s^{(i)})$ is the same as $(n^{-1}X_1,\dotsc,n^{-1}X_s)$.
\end{itemize}
Here one may take $\mc A_n$ to be the reduced free product of $n$ copies of $M_n(\mc A)$ with amalgamation over $\mc B$, see \cite{voi0}.  

For $1 \leq r \leq s$ let $y_r' = y_r^{(1)} + \dotsb + y_r^{(n)}$.  We claim that $(y_1,\dotsc,y_s)$ has the same $\mc B$-valued joint distribution as $(y'_1,\dotsc,y'_s)$.  Indeed, from the proof of Theorem \ref{uniformrcyclictheorem}, for any $1 \leq r_1,\dotsc,r_k \leq n$ and $b_1,\dotsc,b_k \in \mc B$ we have
\begin{align*}
 \kappa^{(k)}[y_{r_1}b_1,\dotsc,y_{r_k}b_k] &= \kappa_E^{(k)}[x_{11}^{(r_1)}b_1,\dotsc,x_{11}^{(r_k)}b_k]\\
&= n^{1-k}\kappa_{E_{\mc B}}^{(k)}[X_{r_1}b_1,\dotsc,X_{r_k}b_k]\\
&= n \cdot \kappa_{E_{\mc B}}^{(k)}[(n^{-1}X_{r_1})b_1,\dotsc,(n^{-1}X_{r_k})b_k].
\end{align*}
On the other hand, since $\{y_1^{(1)},\dotsc,y_s^{(1)}\},\dotsc,\{y_{1}^{(n)},\dotsc,y_s^{(n)}\}$ are free with amalgamation over $\mc B$, we have
\begin{align*}
 \kappa_{E_n}^{(k)}[y'_{r_1}b_1,\dotsc,y'_{r_k}b_k] &= \sum_{i=1}^n \kappa_{E_n}[y_{r_1}^{(i)}b_1,\dotsc,y_{r_k}^{(i)}b_k]\\
&= n\cdot \kappa_{E_{\mc B}}^{(k)}[(n^{-1}X_{r_1})b_1,\dotsc,(n^{-1}X_{r_k})b_k].
\end{align*}
This proves the implication (1)$\Rightarrow$(2).

Conversely, suppose that (2) holds.  By replacing $\mc A_n$ with the reduced free product of $\mc A_n$ and $M_n(\mc B)$, with amalgamation over $\mc B$, we make find a system of matrix units $(e_{ij})_{1 \leq i,j \leq n}$ which commute with $\mc B$ and are free from $\{y_1^{(1)},\dotsc,y_s^{(1)}\}$ with amalgamation over $\mc B$, and such that $E[e_{ij}] = \delta_{ij}n^{-1}$ .  Let $p$ be the projection $e_{11}$ and consider the compressed C$^*$-algebra $p\mc Ap$, with conditional expectation $E_p:p\mc A p \to \mc Bp$ defined by $E_p[pap] = n\cdot E[a]\cdot p$.  Note that $b \mapsto b p$ is a unital inclusion of $\mc B$ into $p\mc A p$.

For $1 \leq i,j \leq n$ and $1 \leq r \leq s$, let $x_{ij}^{(r)} = n\cdot e_{1i}y_r^{(1)}e_{j1} \in p\mc A p$.  For $r =1,\dotsc,s$, let $X_r = (x_{ij}^{(r)})_{1 \leq i, j \leq n}$ in $M_n(p\mc Ap)$.  Let $V_{ij} \in M_n(p\mc A p)$ be the standard system of matrix units, and observe that
\begin{align*}
 E_p[\tr[V_{i_1j_1}X_{r_1}b_1V_{i_2j_2}X_{r_2}b_2\dotsb V_{i_kj_k}X_{r_k}b_k]] &= n^{k}E[e_{1j_1}y_{r_1}^{(1)}e_{i_21}b_1e_{1j_2}y_{r_2}^{(1)}\dotsb e_{1j_k}y_{r_k}^{(1)}V_{i_11}b_k]\cdot p \\
&= n^{k}E[e_{i_1j_1}y_{r_1}^{(1)}b_1e_{i_2j_2}y_{r_2}^{(1)}b_2\dotsb e_{i_kj_k}y_{r_k}^{(1)}b_k]\cdot p
\end{align*}
so that $(X_1,\dotsc,X_s) \cup \{v_{ij}:1 \leq i,j \leq n\}$ has the same $\mc B$-valued joint distribution as $(ny_1^{(1)},\dotsc,ny_s^{(1)}) \cup \{e_{ij}:1 \leq i,j \leq n\}$.  In particular, $(X_1,\dotsc,X_s)$ is free from $M_n(\mc B)$ with amalgamation over $\mc B$, and hence uniformly $R$-cyclic.  Finally, as above we have
\begin{align*}
\kappa_{E_p}^{(k)}[x_{11}^{(r_1)}b_1,\dotsc,x_{11}^{(r_k)}b_k] &= n \cdot \kappa_{E_{\mc B}}^{(k)}[(n^{-1}X_{r_1})b_1,\dotsc,(n^{-1}X_{r_k})b_k]\\
&= n \cdot \kappa_{E}^{(k)}[y_{r_1}^{(1)}b_1,\dotsc,y_{r_k}^{(1)}b_k]\\
&= \kappa_E^{(k)}[y_{r_1}b_1,\dotsc,y_{r_k}b_k].
\end{align*}
So $(y_1,\dotsc,y_s)$ has the same $\mc B$-valued joint distribution as $(x_{11}^{(1)},\dotsc,x_{11}^{(s)})$, which completes the proof.
\end{proof}

Theorem \ref{infdivis} now follows easily.

\begin{proof}[Proof of Theorem \ref{infdivis}]
The implication (1)$\Rightarrow$(2) is immediate from Theorem \ref{finitedivis}.  Suppose then that (2) holds.  By Theorem \ref{finitedivis}, for each $n \in \N$ there is a C$^*$-algebra $\mc A_n$, a unital inclusion $\mc B \hookrightarrow \mc A_n$, a faithful completely positive conditional expectation $E_n:\mc A_n \to \mc B$, and a family $\{x_{ij}^{(r)}(n): 1 \leq i,j \leq n, 1 \leq r \leq s\}$ such that 
\begin{itemize}
\item $(y_1,\dotsc,y_s)$ has the same $\mc B$-valued joint distribution as $(x_{11}^{(1)}(n),\dotsc,x_{11}^{(s)}(n))$.
\item $X_1,\dotsc,X_s$ forms a $\mc B$-valued uniformly $R$-cyclic family, where $X_r = (x_{ij}^{(r)}(n))_{1 \leq i,j \leq n}$ for $1 \leq r \leq s$.
\end{itemize}

Clearly we may assume that, for each $n \in \N$, $\mc A_n$ is generated as a C$^*$-algebra by $\mc B \cup \{x_{ij}^{(r)}(n): 1 \leq i,j \leq n, 1 \leq r \leq s\}$.  Note that $(x_{ij}^{(r)}(n+1))_{1 \leq i,j \leq n, 1 \leq r \leq s}$ has the same $\mc B$-valued joint distribution as $(x_{ij}^{(r)}(n))_{1 \leq i,j \leq n, 1 \leq r \leq s}$.  Since $E_n$ is faithful, for each $n \in \N$ there is a unique unital $*$-homomorphism $\iota_n:\mc A_n \to \mc A_{n+1}$ such that $\iota_n|_{\mc B} = \mathrm{id}_{\mc B}$ and $\iota_n(x_{ij}^{(r)}(n)) = x_{ij}^{(r)}(n+1)$ for $1 \leq i,j \leq n, 1 \leq r \leq s$.  Let $\mc A$ be the inductive limit of this system (see e.g. \cite{black}), and let $j_n:\mc A_n \to \mc A$ be the associated inclusions.  Then $\mc A$ contains $\mc B$ as a unital C$^*$-subalgebra, and there is a unique faithful completely positive conditional expectation $E:\mc A \to \mc B$ such that $E_n[a] = E[j(a_n)]$ for $n \in \N$ and $a \in \mc A_n$.  For $i,j \in \N$ and $1 \leq r \leq s$, let $x_{ij}^{(r)} = j_n(x_{ij}^{(r)}(n))$ for some $n \geq \max\{i,j\}$, it is clear that this does not depend on the choice of $n$.  For $1 \leq r \leq s$ let $X_r = (x_{ij}^{(r)})_{i,j \in \N}$, then it is clear that $(X_1,\dotsc,X_s)$ is a $\mc B$-valued uniformly $R$-cyclic family and that $(y_1,\dotsc,y_s)$ has the same joint distribution as $(x_{11}^{(1)},\dotsc,x_{11}^{(s)})$, which completes the proof.
\end{proof}

\section{Quantum invariant families of matrices}\label{sec:qinv}

Let $G \subset O_n^+$ be a compact orthogonal quantum group.  Define $\alpha:\C \langle t_{ij}^{(r)}: 1\leq i,j \leq n, 1 \leq r \leq s \rangle \to \C \langle t_{ij}^{(r)}: 1 \leq i, j \leq n, 1 \leq r \leq s \rangle \otimes C(G)$ to be the homomorphism determined by
\begin{equation*}
 \alpha(t_{j_1j_2}^{(r)}) = \sum_{1 \leq i_1,i_2 \leq n} t_{i_1i_2}^{(r)} \otimes u_{i_1j_1}u_{i_2j_2}.
\end{equation*}
It is easily checked that $\alpha$ is a \textit{coaction}, which can be thought of as the conjugation action of $G$ on an $s$-tuple of matrices with noncommutative entries.  We will be interested in families of matrices for which the joint distribution of their entries is invariant under conjugation by $G$.  More precisely, we make the following definition:

\begin{definition}
Let $G \subset O_n^+$ be a compact orthogonal quantum group, and let $(\mc A,\varphi)$ be a noncommutative probability space.  Let $X_1,\dotsc,X_s$ be a family of matrices in $M_n(\mc A)$, $X_r = (x_{ij}^{(r)})_{1 \leq i,j \leq n}$.  We say that the joint distribution of $(x_{ij}^{(r)})$ is \textit{invariant under conjugation by $G$}, or that the family $X_1,\dotsc,X_s$ is \textit{$G$-invariant}, if
\begin{equation*}
 (\varphi_x \otimes \mathrm{id})\alpha(p) = \varphi_x(p)\cdot 1_{C(G)}
\end{equation*}
for any $p \in \C \langle t_{ij}^{(r)}: 1 \leq i, j \leq n, 1 \leq r \leq s \rangle$.
\end{definition}

\begin{remark}\hfill
\begin{enumerate}
\item Explicitly, the condition is that
\begin{equation*}
 \varphi(x_{j_{11}j_{12}}^{(r_1)}\dotsb x_{j_{k1}j_{k2}}^{(r_k)})\cdot 1_{C(G)} = \sum_{1 \leq i_{11},i_{12},\dotsc,i_{k2} \leq n} \varphi(x_{i_{11}i_{12}}^{(r_1)}\dotsb x_{i_{k1}i_{k2}}^{(r_k)})\cdot u_{i_{11}j_{11}}u_{i_{12}j_{12}} \dotsb u_{i_{k2}j_{k2}}
\end{equation*}
for any $1 \leq j_{11},j_{12},\dotsc,j_{k2} \leq n$ and $1 \leq r_1,\dotsc,r_k \leq s$.
\item If $G \subset O_n$ is a compact orthogonal group, evaluating both sides of the above equation at $g \in G$ yields
\begin{equation*}
 \varphi(x_{j_{11}j_{12}}^{(r_1)}\dotsb x_{j_{k1}j_{k2}}^{(r_k)}) = \varphi\bigl( (g^tX_{r_1}g)_{j_{11}j_{12}}\dotsb (g^tX_{r_k}g)_{j_{k1}j_{k2}}\bigr),
\end{equation*}
so that we recover the usual invariance condition.
\end{enumerate}

\end{remark}

We will give a relation between $G$-invariance and the ``easiness'' condition for a compact orthogonal quantum group $G$ in Theorem \ref{easyinv} below.  The first observation is that $G$-invariance of a family of matrices is equivalent to $G$-invariance of the moment series of their entries.

\begin{lemma}\label{invfix}
Let $G \subset O_n^+$ be a compact orthogonal quantum group, and let $(\mc A,\varphi)$ be a noncommutative probability space.  A family $X_1,\dotsc,X_s$ in $M_n(\mc A)$ is $G$-invariant if and only if for each $k \in \N$ and $1 \leq r_1,\dotsc,r_k \leq s$, the vector
\begin{equation*}
 \sum_{1 \leq i_{11},\dotsc,i_{k2} \leq n} \varphi(x_{i_{11}i_{12}}^{(r_1)}\dotsb x_{i_{k1}i_{k2}}^{(r_k)})\cdot e_{i_{11}} \otimes e_{i_{12}}\otimes \dotsb \otimes e_{i_{k2}} \in (\C^n)^{\otimes 2k}
\end{equation*}
is fixed by $u^{\otimes 2k}$, where $u$ is the fundamental representation of $G$.  
\end{lemma}

\begin{proof}
Let $\Psi:C(G) \to C(G)$ be the automorphism given by $\Psi(f) = S(f)^*$.  Let $\theta_k$ denote the vector in the statement of the proposition, then we have
\begin{equation*}
(\mathrm{id} \otimes \Psi) u^{\otimes 2k}(\theta_k) = \negthickspace\sum_{\substack{1 \leq j_{11},\dotsc,j_{k2} \leq n\\ 1 \leq i_{11},\dotsc,i_{k2} \leq n}} \negthickspace\varphi(x_{i_{11}i_{12}}^{(r_1)}\dotsb x_{i_{k1}i_{k2}}^{(r_k)})\cdot e_{j_{11}} \otimes e_{j_{12}} \otimes \dotsb \otimes e_{j_{k2}} \otimes u_{i_{11}j_{11}}u_{i_{12}j_{12}}\dotsb u_{i_{k2}j_{k2}}.
\end{equation*}
Since $\Psi$ is an automorphism, $\theta_k$ is fixed by $u^{\otimes 2k}$ if and only if the above expression is equal to
\begin{equation*}
 \theta_k \otimes 1_{C(G)} = \sum_{1 \leq j_{11},\dotsc,j_{k2} \leq n} \varphi(x_{j_{11}j_{12}}^{(r_1)}\dotsb x_{j_{k1}j_{k2}}^{(r_k)})\cdot e_{j_{11}} \otimes e_{j_{12}}\otimes \dotsb \otimes e_{j_{k2}} \otimes 1_{C(G)}.
\end{equation*}
Equating coefficients on $e_{j_{11}} \otimes e_{j_{12}} \otimes \dotsb \otimes e_{j_{k2}}$ completes the proof.
\end{proof}

\begin{theorem}\label{easyinv}
Let $G$ be a free quantum group $O_n^+,S_n^+,B_n^+$ or $H_n^+$ with associated partitions $D(k) \subset NC(k)$.  Let $X_1,\dotsc,X_s$ be a family of matrices in $M_n(\mc A)$, $X_r = (x_{ij}^{(r)})_{1 \leq i,j \leq n}$.  Then $X_1,\dotsc,X_s$ is $G$-invariant if and only if for each $k \in \N$ and $1 \leq r_1,\dots,r_k \leq s$ there is a collection of numbers $\{c_{\pi,\mathbf r}: \pi \in D(2k)\}$ such that
\begin{equation*}
 \varphi(x_{i_{11}i_{12}}^{(r_1)}\dotsb x_{i_{k1}i_{k2}}^{(r_k)}) = \sum_{\substack{\pi \in D(2k)\\ \pi \leq \ker \mathbf i}} c_{\pi,\mathbf r}
\end{equation*}
for any $1 \leq i_{11},\dotsc,i_{k2} \leq n$.
\end{theorem}

\begin{proof}
By Lemma \ref{invfix}, $X_1,\dotsc,X_s$ is $G$-invariant if and only if for each $k \in \N$ and $1 \leq r_1,\dotsc,r_k \leq s$ the vector
\begin{equation*}
\sum_{1 \leq i_{11},\dotsc,i_{k2} \leq n} \varphi(x_{i_{11}i_{12}}^{(r_1)}\dotsb x_{i_{k1}i_{k2}}^{(r_k)}) \cdot e_{i_{11}}\otimes e_{i_{12}}\otimes \dotsb \otimes e_{i_{k2}}
\end{equation*}
is fixed by $u^{\otimes 2k}$.  But recall that $Fix(u^{\otimes 2k}) = \mathrm{span} \{ T_\pi: \pi \in D(2k)\}$, where
\begin{equation*}
 T_\pi = \sum_{\substack{1 \leq i_{11},\dotsc,i_{k2} \leq n \\ \pi \leq \ker \mathbf i}} e_{i_{11}} \otimes e_{i_{12}}\otimes \dotsb \otimes e_{i_{k2}}.
\end{equation*}
It follows that $X$ is $G$-invariant if and only if for each $k \in \N$ and $1\leq r_1,\dotsc, r_k \leq s$ there are numbers $\{c_{\pi,\mathbf r}: \pi \in D(2k)\}$ such that
\begin{equation*}
 \sum_{1 \leq i_{11},\dotsc,i_{k2} \leq n} \varphi(x_{i_{11}i_{12}}^{(r_1)}\dotsb x_{i_{k1}i_{k2}}^{(r_k)}) \cdot e_{i_{11}} \otimes e_{i_{12}}\otimes \dotsb \otimes e_{i_{k2}} = \sum_{\pi \in D(2k)} c_{\pi,\mathbf r}\cdot T_\pi.
\end{equation*}
Equating coefficients on $e_{i_{11}} \otimes \dotsb \otimes e_{i_{k2}}$ yields the desired result.
\end{proof}

\begin{remark}
In the next section, we will see that if $X_1,\dotsc,X_s$ is a uniformly $R$-cyclic family with respect to a $\varphi$-preserving conditional expectation $E$, then $\{c_{\pi,\mathbf r}:\pi \in NC_2(2k)\}$ can be taken to be $\varphi$ applied to certain ``cyclic'' operator-valued cumulants, which establishes $O_n^+$-invariance.  Theorem \ref{orthogonalcase} can be viewed as a kind of converse: if there are $\{c_{\pi,\mathbf r}:\pi \in NC_2(2k)\}$ which satisfy the relation above for a self-adjoint family $X_1,\dotsc,X_s$ of infinite matrices with entries in a W$^*$-probability space $(M,\varphi)$, then they must be given by $\varphi$ applied to cyclic operator-valued cumulants.  The statement for $H_n^+$-invariance is more complicated but of a similar nature, see Proposition \ref{h+finite}.  In general these $c_\pi$ appear to be rather mysterious, a better understanding here might help with characterizing $S^+$ and $B^+$-invariant matrices.
\end{remark}

To further analyze the structure of $G$-invariant families, we will need a more analytic framework.  Throughout the rest of the section, $(M,\varphi)$ will be a W$^*$-probability space.  $X_1,\dotsc,X_s$ will be a family of matrices in $M_n(M)$, $X_r = (x_{ij}^{(r)})_{1 \leq i,j \leq n}$, which is \textbf{self-adjoint} in the sense that whenever $X$ is in the family, so is $X^*$.  In other words there is an involution $\sigma$ of $\{1,\dotsc,s\}$ such that $X_r^* = X_{\sigma(r)}$.  Observe that the coaction $\alpha$ is a $*$-homomorphism when $\C \langle t_{ij}^{(r)}: 1\leq i,j \leq n, 1 \leq r \leq s\rangle$ is given the $*$-structure determined by ${t_{ij}^{(r)}}^* = t_{ji}^{(\sigma(r))}$.

By restricting if necessary, we will assume that $M$ is generated by $\{x_{ij}^{(r)}:1 \leq i,j \leq n, 1 \leq r \leq s\}$.  Since $\alpha$ preserves the $*$-distribution of $(x_{ij}^{(r)})$, it follows that $\alpha$ extends to a coaction $\overline \alpha:M \to M \overline \otimes L^\infty(G)$ determined by
\begin{equation*}
\overline \alpha(p(x)) = (\mathrm{ev}_x \otimes \pi) \alpha(p)
\end{equation*}
for $p \in \C \langle t_{ij}^{(r)}: 1 \leq i,j \leq n, 1\leq r \leq s \rangle$, where $L^\infty(G)$ denotes the weak closure of $C(G)$ under the GNS representation $\pi$ for the Haar state $\int$, see e.g. \cite{cur3}.  Let $\mc B$ denote the fixed point algebra,
\begin{equation*}
 \mc B = \{m \in M: \overline \alpha(m) = m \otimes 1 \},
\end{equation*}
then $E = (\mathrm{id} \otimes \int) \circ \overline \alpha$ defines a $\varphi$-preserving conditional expectation of $M$ onto $\mc B$.  We will now give expressions for the $\mc B$-valued moment functionals in the case that $G$ is a free quantum group.  

\begin{theorem}\label{finitemoments}
Suppose that $G \subset O_n^+$ is a free quantum group with associated partitions $D(k) \subset NC(k)$.  Let $\tau \in NC(k)$, $1 \leq j_{11},j_{12},\dotsc,j_{k2} \leq n$, $1 \leq r_1,\dotsc,r_k \leq s$ and $b_0,\dotsc,b_k \in \mc B$, then
\begin{multline*}
 E^{(\tau)}[b_0x_{j_{11}j_{12}}^{(r_1)}b_1,\dotsc,x_{j_{k1}j_{k2}}^{(r_k)}b_k]\\
 = \sum_{\substack{\sigma \in D(2k)\\ \sigma \leq \hat \tau \wedge \ker \mathbf j}} \sum_{\substack{\pi \in D(2k)\\ \pi \leq \hat \tau}} \biggl(\prod_{V \in \hat \tau} W_{D(V),n}(\pi|_V,\sigma|_V)\biggr) \sum_{\substack{1 \leq i_{11},\dotsc,i_{k_2}\leq n\\ \pi \leq \ker \mathbf i}} b_0x_{i_{11}i_{12}}^{(r_1)}\dotsb x_{i_{k1}i_{k2}}^{(r_k)}b_k,
\end{multline*}
where $\pi|_V,\sigma|_V$ denote the restrictions of $\pi,\sigma$ to $V$.
\end{theorem}

\begin{proof}
First consider the case $\tau = 1_k$ is the partition with only one block.  Since $b_0,\dotsc,b_k$ are fixed by $\overline \alpha$, we have
\begin{equation*}
 \overline \alpha(b_0x_{j_{11}j_{12}}^{(r_1)}\dotsb x_{j_{k1}j_{k2}}^{(r_k)}b_k) = \sum_{1 \leq i_{11},\dotsc,i_{k2} \leq n} b_0x_{i_{11}i_{12}}^{(r_1)}\dotsb x_{i_{k1}i_{k2}}^{(r_k)}b_k \otimes u_{i_{11}j_{11}}\dotsb u_{i_{k2}j_{k2}}.
\end{equation*}
Then
\begin{align*}
 E[b_0x_{j_{11}j_{12}}^{(r_1)}b_1\dotsb x_{j_{k1}j_{k2}}^{(r_k)}b_k] &= \sum_{1 \leq i_{11},\dotsc,i_{k2} \leq n} b_0x_{i_{11}i_{12}}^{(r_1)}\dotsb x_{i_{k1}i_{k2}}^{(r_k)}b_k \cdot \int_{G} u_{i_{11}j_{11}}\dotsb u_{i_{k2}j_{k2}}\\
&= \sum_{1 \leq i_{11},\dotsc,i_{k2} \leq n} b_0x_{i_{11}i_{12}}^{(r_1)}\dotsb x_{i_{k1}i_{k2}}^{(r_k)}b_k \cdot \bigl(\sum_{\substack{\pi,\sigma \in D(2k)\\ \pi \leq \ker \mathbf i\\ \sigma \leq \ker \mathbf j}} W_{D(k),n}(\pi,\sigma)\bigr)\\
&= \sum_{\substack{\sigma \in D(2k)\\ \sigma \leq \ker \mathbf j}} \sum_{\pi \in D(2k)} W_{D(2k),n}(\pi,\sigma)\sum_{\substack{1 \leq i_{11},\dotsc,i_{k2} \leq n\\ \pi \leq \ker \mathbf i}} b_0x_{i_{11}i_{12}}^{(r_1)}\dotsb x_{i_{k1}i_{k2}}^{(r_k)}b_k,
\end{align*}
as claimed.  The general case follows from induction on the number of blocks of $\tau$.  
\end{proof}

\subsection{Infinite quantum invariant families} \label{sec:infinvar}\hfill

Suppose now that $G$ is one the free quantum groups $O,S,B,H$.  Note that for $n < m$ we have inclusions $G_n \hookrightarrow G_m$, expressed as the Hopf algebra morphisms $\omega_{n,m}:C(G_m) \to C(G_n)$ defined by
\begin{equation*}
 \omega_{n,m}(u_{ij}) = \begin{cases} u_{ij}, & 1 \leq i, j \leq n\\ \delta_{ij}, & \max\{i,j\} > n\end{cases}.
\end{equation*}

A self-adjoint family $X_1,\dotsc,X_s$ of infinite matrices, $X_r = (x_{ij}^{(r)})_{i,j \in \N}$, will be called \textit{$G$-invariant} if for each $n \in \N$ the family $X_1^{(n)},\dotsc,X_s^{(n)}$, $X_r^{(n)} = (x_{ij}^{(r)})_{1 \leq i, j\leq n}$, is $G_n$-invariant.  For working with infinite matrices, it will be convenient to modify the coactions defined above as follows.  For each $n \in \N$, let $\beta_n:\C \langle t_{ij}^{(r)} :i,j \in \N, 1\leq r \leq s \rangle \to \C \langle t_{ij}^{(r)}:i,j \in \N, 1 \leq r \leq s\rangle \otimes C(G_n)$ be the unital homomorphism determined by
\begin{equation*}
 \beta_n(t_{j_{1}j_2}^{(r)}) = \begin{cases} \sum_{1 \leq i_1,i_2 \leq n} t_{i_1i_2}^{(r)} \otimes u_{i_1j_1}u_{i_2j_2}, & 1 \leq j_1,j_2 \leq n\\
                          \sum_{1 \leq i_1 \leq n} t_{i_1j_1}^{(r)} \otimes u_{i_1j_1}, & 1 \leq j_1 \leq n < j_2\\
\sum_{1 \leq i_2 \leq n} t_{i_2j_2}^{(r)} \otimes u_{i_2j_2}, & 1 \leq j_2 \leq n < j_1\\
t_{j_1j_2}^{(r)} \otimes 1_{C(G_n)}, & n < \min\{j_1,j_2\}
                         \end{cases}.
\end{equation*}
It is easily verified that $\beta_n$ is a coaction.  Moreover, we have the compatibilities
\begin{equation*}
(\mathrm{id} \otimes \omega_{n,m})\circ \beta_m = \beta_n
\end{equation*}
and
\begin{equation*}
 (\iota_n \otimes \mathrm{id}) \circ \alpha_n = \beta_n \circ \iota_n,
\end{equation*}
where $\iota_n:\C \langle t_{ij}^{(r)}: 1 \leq i,j \leq n, 1 \leq r \leq s \rangle \to \C \langle t_{ij}^{(r)}:i,j \in \N, 1 \leq r \leq s\rangle$ is the obvious inclusion.
Using these compatibilities, it is not hard to see that a family $X_1,\dotsc,X_s$ is $G$-invariant if and only if $\varphi_x$ is invariant under the coactions $\beta_n$ for each $n \in \N$.

Suppose now that $X_1,\dotsc,X_s$ is a self-adjoint $G$-invariant family of infinite matrices random variables in $(M,\varphi)$, and assume that $M$ is generated by $\{x_{ij}^{(r)}:i,j \in \N, 1\leq r \leq s\}$.  As above, the coactions $\beta_n$ extend to $\overline \beta_n:M \to M \overline\otimes L^\infty(G_n)$.  Let $\mc B_n$ be the fixed point algebra of $\overline \beta_n$, and let $E_n = (\mathrm{id} \otimes \int) \circ \overline \beta_n:M \to \mc B_n$ be the $\varphi$-preserving conditional expectation given by integrating the action of $G_n$.  The advantage of using $\beta_n$ is that the fixed point algebras $\mc B_n$ are now nested, which follows from $\beta_{n} = (\mathrm{id} \otimes \omega_{n,n+1}) \circ \beta_{n+1}$.  Define the \textit{$G$-invariant subalgebra} $\mc B$ by
\begin{equation*}
 \mc B = \bigcap_{n \geq 1} \mc B_n.
\end{equation*}
A simple reversed martingale convergence argument shows that there is a $\varphi$-preserving conditional expectation $E:M \to \mc B$ given by
\begin{equation*}
 E[m] = \lim_{n \to \infty} E_n[m],
\end{equation*}
where the limit is taken in the strong operator topology, see e.g. \cite[Proposition 4.7]{cur4}.

We can now give formulas for the moment and cumulant functionals taken with respect to the $G$-invariant subalgebra.

\begin{theorem}\label{limitcumulants}
Let $G$ be one of $O^+,S^+,B^+,H^+$, with associated partitions $D(k) \subset NC(k)$.  Let $\tau \in NC(k)$, $j_{11},j_{12},\dotsc,j_{k2} \in \N$ and $b_0,\dotsc,b_k \in B$.  Then we have
\begin{equation*}
E^{(\tau)}[b_0x_{j_{11}j_{12}}^{(r_1)}b_1,\dotsc,x_{j_{k1}j_{k2}}^{(r_k)}b_k] = \lim_{n \to \infty} \sum_{\substack{\pi,\sigma \in D(2k)\\ \pi \leq \sigma \leq \hat \tau \wedge \ker \mathbf j}} \mu(\pi,\sigma)n^{-|\pi|}\sum_{\substack{1 \leq i_{11},\dotsc,i_{k2} \leq n\\ \pi \leq \ker \mathbf i}} b_0x_{i_{11}i_{12}}^{(r_1)}\dotsb x_{i_{k1}i_{k2}}^{(r_k)}b_k
\end{equation*}
and
\begin{equation*}
 \kappa^{(\tau)}[b_0x_{j_{11}j_{12}}^{(r_1)}b_1,\dotsc,x_{j_{k1}j_{k2}}^{(r_k)}b_k] = \lim_{n \to \infty} \sum_{\substack{\sigma \in D(2k)\\ \sigma \leq \ker \mathbf j\\ \sigma \vee \hat 0_k = \hat \tau}} \sum_{\substack{\pi \in D(2k)\\ \pi \leq \sigma}} \mu(\pi,\sigma)n^{-|\pi|}\negthickspace\sum_{\substack{1 \leq i_{11},\dotsc,i_{k2} \leq n\\ \pi \leq \ker \mathbf i}}\negthickspace b_0x_{i_{11}i_{12}}^{(r_1)}\dotsb x_{i_{k1}i_{k2}}^{(r_k)}b_k,
\end{equation*}
where the limits are taken in the strong operator topology.
\end{theorem}

\begin{proof}
We will first prove the formula for the moment functionals.  By a reversed martingale convergence argument we have
\begin{equation*}
 E^{(\tau)}[b_0x_{j_{11}j_{12}}^{(r_1)}b_1,\dotsc,x_{j_{k1}j_{k2}}^{(r_k)}b_k] = \lim_{n \to \infty} E_n^{(\tau)}[b_0x_{j_{11}j_{12}}^{(r_1)}b_1,\dotsc,x_{j_{k1}j_{k2}}^{(r_k)}b_k]
\end{equation*}
with convergence in the strong topology, see e.g. \cite[Proposition 4.7]{cur4}.  From Proposition \ref{finitemoments}, the right hand side is equal to
\begin{equation*}
 \lim_{n \to \infty} \sum_{\substack{\sigma \in D(2k)\\ \sigma \leq \hat \tau \wedge \ker \mathbf j}} \sum_{\substack{\pi \in D(2k)\\ \pi \leq \hat \tau}} \biggl(\prod_{V \in \hat \tau} W_{D(V),n}(\pi|_V,\sigma|_V)\biggr) \sum_{\substack{1 \leq i_{11},\dotsc,i_{k_2}\leq n\\ \pi \leq \ker \mathbf i}} b_0x_{i_{11}i_{12}}^{(r_1)}\dotsb x_{i_{k1}i_{k2}}^{(r_k)}b_k.
\end{equation*}
Now since the family is $S^+$-invariant and hence $S$-invariant, and is self-adjoint, it follows that the $*$-distribution of $x^{(r)}_{ij}$ depends only on $r$ and on whether $i$ is equal to $j$.  Since $\varphi$ is faithful, it follows that there is a finite constant $C$ such that $\|x^{(r)}_{ij}\| \leq C$ for $1 \leq r \leq s$ and all $i,j \in \N$.  We then have
\begin{equation*}
 \biggl\|\sum_{\substack{1 \leq i_{11},\dotsc,i_{k_2}\leq n\\ \pi \leq \ker \mathbf i}} b_0x_{i_{11}i_{12}}^{(r_1)}\dotsb x_{i_{k1}i_{k2}}^{(r_k)}b_k\biggr\| \leq n^{|\pi|} C^k \|b_0\|\dotsb \|b_k\|.
\end{equation*}
Now from Theorem \ref{West}, if $\pi, \sigma \leq \hat \tau$ are in $D(2k)$ then we have
\begin{equation*}
n^{|\pi|} \biggl(\prod_{V \in \hat \tau} W_{D(V),n}(\pi|_V,\sigma|_V)\biggr) = \prod_{V \in \hat \tau}\bigl( \mu_V(\pi|_V,\sigma|_V) + O(n^{-1})\bigr)= \mu(\pi,\sigma) + O(n^{-1}),
\end{equation*}
where we have used the multiplicativity of the M\"{o}bius function on $NC(2k)$.  Combining these equations yields the desired result.

The statement for cumulants now follows from M\"{o}bius inversion.  Indeed, we have
\begin{multline*}
 \kappa^{(\tau)}[b_0x_{j_{11}j_{12}}^{(r_1)}b_1,\dotsc,x_{j_{k1}j_{k2}}^{(r_k)}b_k] = \sum_{\substack{\rho \in NC(k)\\ \rho \leq \tau}} \mu(\rho,\tau)E^{(\rho)}[b_0x_{j_{11}j_{12}}^{(r_1)}b_1,\dotsc,x_{j_{k1}j_{k2}}^{(r_k)}b_k]\\
= \lim_{n \to \infty} \sum_{\substack{\rho \in NC(k)\\ \rho \leq \tau}} \mu(\rho,\tau)\sum_{\substack{\pi,\sigma \in D(2k)\\ \pi \leq \sigma \leq \hat \rho \wedge \ker \mathbf j}} \mu(\pi,\sigma)n^{-|\pi|}\sum_{\substack{1 \leq i_{11},\dotsc,i_{k2} \leq n\\ \pi \leq \ker \mathbf i}} b_0x_{i_{11}i_{12}}^{(r_1)}\dotsb x_{i_{k1}i_{k2}}^{(r_k)}b_k\\
= \lim_{n \to \infty} \sum_{\substack{\pi,\sigma \in D(2k)\\ \pi \leq \sigma \leq \hat \tau \wedge \ker \mathbf j}} \biggl(\sum_{\substack{\rho \in NC(k)\\ \sigma \leq \hat \rho \leq \hat \tau}} \mu(\rho,\tau)\biggr) \mu(\pi,\sigma)n^{-|\pi|}\sum_{\substack{1 \leq i_{11},\dotsc,i_{k2} \leq n\\ \pi \leq \ker \mathbf i}} b_0x_{i_{11}i_{12}}^{(r_1)}\dotsb x_{i_{k1}i_{k2}}^{(r_k)}b_k.
\end{multline*}
Note that every non-crossing partition in the interval $(\sigma \vee \hat 0_k, \hat \tau)$ is of the form $\hat \rho$ for a unique $\rho \in NC(k)$, and moreover we have $\mu(\hat \rho, \hat \tau) = \mu(\rho,\tau)$.  It follows that
\begin{equation*}
\sum_{\substack{\rho \in NC(k)\\ \sigma \leq \hat \rho \leq \hat \tau}} \mu(\rho,\tau) = \sum_{\substack{\theta \in NC(2k)\\ \sigma \vee \hat 0_k \leq \theta \leq \hat \tau}} \mu(\theta,\hat \tau) = \begin{cases} 1, & \sigma \vee \hat 0_k = \hat \tau\\ 0, & \text{otherwise}\end{cases},
\end{equation*}
from which the result follows.
\end{proof}

\begin{remark}\label{passage}
In general it is not clear how to simplify the expression for cumulants given in Theorem \ref{limitcumulants} above.  The difficulty is that on the left hand side we have cumulants indexed by non-crossing partition on $k$ points, while the right hand side is expressed in terms of partitions of $2k$ points.  In the next two sections, we will show that in the free orthogonal and hyperoctahedral cases the corresponding partitions $D(2k)$ can be reexpressed in terms of partitions in $NC(k)$.  This will allow us to further analyze the $\mc B$-valued cumulants, and prove Theorems \ref{orthogonalcase} and \ref{hyperoctahedralcase}.
\end{remark}

\section{The free orthogonal case}\label{sec:o+case}

In this section we will complete the proof of Theorem \ref{orthogonalcase}.  As discussed in Remark \ref{passage} above, to further analyze the cumulant formula in Theorem \ref{limitcumulants} we will need to use the fattening procedure to connect $NC(k)$ with $NC_2(2k)$.  

We first prove the implication $(2) \Rightarrow (1)$ in Theorem \ref{orthogonalcase}.  This in fact holds for finite matrices and in a purely algebraic setting.

\begin{proposition}
Let $(\mc A,\varphi)$ be a noncommutative probability space and let $X_1,\dotsc,X_s$ be a family of matrices in $M_n(\mc A)$, $X_r = (x_{ij}^{(r)})_{1 \leq i,j \leq n}$.  Suppose that there is a subalgebra $1 \in \mc B \subset \mc A$ and a $\varphi$-preserving conditional expectation $E:\mc A \to \mc B$ such that $X_1,\dotsc,X_s$ is uniformly $R$-cyclic with respect to $E$.  Then the family $X_1,\dotsc,X_s$ is $O_n^+$-invariant.
\end{proposition}

\begin{proof}
Let $1 \leq i_{11},i_{12},\dotsc,i_{k2} \leq n$ and $1 \leq r_1,\dotsc,r_k \leq s$, then we have
\begin{align*}
 \varphi(x_{i_{11}i_{12}}^{(r_1)}\dotsb x_{i_{k1}i_{k2}}^{(r_k)}) &= \varphi(E[x_{i_{11}i_{12}}^{(r_1)}\dotsb x_{i_{k1}i_{k2}}^{(r_k)}])\\
&= \sum_{\pi \in NC(k)} \varphi(\kappa_E^{(\pi)}[x_{i_{11}i_{12}}^{(r_1)},\dotsc,x_{i_{k1}i_{k2}}^{(r_k)}]).
\end{align*}
Now recall from Section \ref{sec:rcyclic} that
\begin{equation*}
 \kappa_E^{(\pi)}[x_{i_{11}i_{12}}^{(r_1)},\dotsc,x_{i_{k1}i_{k2}}^{(r_k)}] = \begin{cases} \kappa_E^{(\pi)}[x_{11}^{(r_1)},\dotsc,x_{11}^{(r_k)}], & \widetilde \pi \leq \ker \mathbf i\\ 
                                                               0, &\text{otherwise}
                                                              \end{cases},
\end{equation*}
so we have
\begin{equation*}
 \varphi(x_{i_{11}i_{12}}^{(r_1)}\dotsb x_{i_{k1}i_{k2}}^{(r_k)}) = \sum_{\substack{\pi \in NC(k)\\ \widetilde \pi \leq \ker \mathbf i}} \varphi(\kappa_E^{(\pi)}[x_{11}^{(r_1)},\dotsc,x_{11}^{(r_k)}])
\end{equation*}

Setting $c_{\widetilde \pi,\mathbf r} = \varphi(\kappa_E^{(\pi)}[x_{11}^{(r_1)},\dotsc,x_{11}^{(r_k)}])$ for $\pi \in NC(k)$, and replacing the sum over $\pi \in NC(k)$ by the sum over $\widetilde \pi \in NC_2(2k)$, we see that $X_1,\dotsc,X_s$ satisfy the criterion for $O_n^+$-invariance given in Theorem \ref{easyinv}.
\end{proof}

We now complete the proof of Theorem \ref{orthogonalcase}.

\begin{proof}[Proof of Theorem \ref{orthogonalcase}]
It remains only to show the implication (1) $\Rightarrow$ (2).  Let $\mc B$ be the $O^+$-invariant subalgebra introduced in Section \ref{sec:qinv}.  Let $b_0,\dotsc,b_k \in \mc B$, $i_{11},\dotsc,i_{k2} \in \N$ and $1 \leq r_1,\dotsc,r_k \leq s$.  From Theorem \ref{limitcumulants} we have
\begin{equation*}
 \kappa_E^{(k)}[b_0x_{j_{11}j_{12}}^{(r_1)}b_1,\dotsc,x_{j_{k1}j_{k2}}^{(r_k)}b_k] = \lim_{n \to \infty} \sum_{\substack{\sigma \in NC_2(2k)\\ \sigma \leq \ker \mathbf j\\ \sigma \vee \hat 0_k = 1_{2k}}} n^{-|\sigma|}\sum_{\substack{1 \leq i_{11},\dotsc,i_{k2} \leq n\\ \sigma \leq \ker \mathbf i}} b_0x_{i_{11}i_{12}}^{(r_1)}\dotsb x_{i_{k1}i_{k2}}^{(r_k)}b_k,
\end{equation*}
note that we have simplified the formula by using the fact that $\pi \leq \sigma \Rightarrow \pi = \sigma$ for $\pi,\sigma \in NC_2(2k)$.  Now the only $\sigma \in NC_2(2k)$ which satisfies $\sigma \vee \hat 0_k = \hat 1_{2k}$ is $\widetilde 1_k$, so that
\begin{equation*}
\kappa_E^{(k)}[b_0x_{j_{11}j_{12}}^{(r_1)}b_1,\dotsc,x_{j_{k1}j_{k2}}^{(r_k)}b_k] = 0
\end{equation*}
unless $\widetilde 1_k \leq \ker \mathbf j$, in which case we have
\begin{equation*}
\kappa_E^{(k)}[b_0x_{j_{11}j_{12}}^{(r_1)}b_1,\dotsc,x_{j_{k1}j_{k2}}^{(r_k)}b_k] = \lim_{n \to \infty} n^{-k}\sum_{1 \leq i_1,\dotsc,i_k \leq n} b_0x_{i_{1}i_{2}}^{(r_1)}b_1x_{i_2i_3}^{(r_2)}\dotsb x_{i_{k}i_{1}}^{(r_k)}b_k.
\end{equation*}
Since the right hand side does not depend on the indices $j_{11},\dotsc,j_{k2}$, we have
\begin{equation*}
 \kappa_E^{(k)}[b_0x_{j_{11}j_{12}}^{(r_1)}b_1,\dotsc,x_{j_{k1}j_{k2}}^{(r_k)}b_k] = \begin{cases} \kappa_E^{(k)}[b_0x_{11}^{(r_1)}b_1,\dotsc,x_{11}^{(r_k)}b_k], & \widetilde 1_k \leq \ker \mathbf j\\
                                                                         0, &\widetilde 1_k \not\leq \ker \mathbf j
                                                                        \end{cases},
\end{equation*}
so that $X_1,\dotsc,X_s$ form a uniformly $R$-cyclic family with respect to $E$ as claimed.
\end{proof}

\section{The free hyperoctahedral case}\label{sec:h+case}

In this section we will consider $H^+$-invariant families and prove Theorem \ref{hyperoctahedralcase}.  First let us give a rigorous definition for the determining series $\theta_X$ of an $R$-cyclic family $X_1,\dotsc,X_s$ to be invariant under quantum permutations.  

\begin{definition}
Let $(\mc A, E:\mc A \to \mc B)$ be an operator-valued probability space, and let $X_1,\dotsc,X_s$ be a $\mc B$-valued $R$-cyclic family of matrices in $M_n(\mc A)$, $X_r = (x_{ij}^{(r)})_{1 \leq i,j \leq n}$.  We say that the $\mc B$-valued determining series of $X_1,\dotsc,X_s$ is \textit{invariant under quantum permutations} if $\theta_X$ is invariant under the coaction $\alpha: \mc B\langle t_i^{(r)}:1 \leq i \leq n, 1 \leq r \leq s \rangle \to \mc B \langle t_i^{(r)}: 1 \leq i \leq n, 1 \leq r \leq s \rangle \otimes C(S_n^+)$ determined by
\begin{align*}
\alpha(b) &= b \otimes 1_{C(S_n^+)}, &  &(b \in \mc B)\\
\alpha(t_j^{(r)}) &= \sum_{i=1}^n t_{i}^{(r)} \otimes u_{ij}, & &(1 \leq j \leq n, 1 \leq r \leq s).
\end{align*}
Explicitly, we require that for any $k \in \N$, $1 \leq j_1,\dotsc,j_k \leq n$, $1 \leq r_1,\dotsc,r_k \leq s$ and $b_1,\dotsc,b_k \in \mc B$ we have
\begin{equation*}
 \sum_{1 \leq i_1,\dotsc,i_k \leq n} \kappa_E^{(k)}[x_{i_ki_1}^{(r_1)}b_1,\dotsc,x_{i_{k-1}i_k}^{(r_k)}b_k] \otimes u_{i_1j_1}\dotsb u_{i_kj_k} = \kappa_{E}^{(k)}[x_{j_kj_1}^{(r_1)}b_1,\dotsc,x_{j_{k-1}j_k}^{(r_k)}b_k] \otimes 1_{C(S_n^+)}
\end{equation*}
as an equality in $\mc B \otimes C(S_n^+)$. 
\end{definition}

\begin{lemma}\label{detserinv}
Let $(\mc A,E:\mc A \to \mc B)$ be an operator-valued probability space, and let $X_1,\dotsc,X_s$ be a $\mc B$-valued $R$-cyclic family in $M_n(\mc A)$, $X_r = (x_{ij}^{(r)})_{1 \leq i,j \leq n}$.  Then the determining series of $X_1,\dotsc,X_s$ is invariant under quantum permutations if and only if for every $k \in \N$, $1 \leq r_1,\dotsc,r_k \leq s$ and $\sigma \in NC(k)$ there are $\C$-multilinear maps $c_{\sigma,\mathbf r}:\mc B^k \to \mc B$ such that
\begin{equation*}
 \kappa_E^{(k)}[x_{i_{11}i_{12}}^{(r_1)}b_1,\dotsc,x_{i_{k1}i_{k2}}^{(r_k)}b_k] = \sum_{\substack{\sigma \in NC(k)\\ \widetilde \sigma \vee \widetilde{1_k} \leq \ker \mathbf i}} c_{\sigma,\mathbf r}[b_1,\dotsc,b_k].
\end{equation*}

\end{lemma}

\begin{proof}
First use Lemma \ref{fatfacts} to find that
\begin{equation*}
 \widetilde \sigma \vee \widetilde{1_k} = \overrightarrow{\overleftarrow{\widetilde \sigma} \vee \overleftarrow{\widetilde{1_k}}} = \overrightarrow{\widetilde{K(\sigma)} \vee \widetilde 0_k} = \overrightarrow{\widehat{K(\sigma)}}.
\end{equation*}
In particular,
\begin{equation*}
\widetilde \sigma \vee \widetilde{1_k} \leq \ker (i_k,i_1,i_1,i_2,\dotsc,i_{k-1},i_k) \Leftrightarrow K(\sigma) \leq \ker (i_1,\dotsc,i_k).
\end{equation*}

Now suppose that there are multilinear maps $c_\sigma$ as in the statement of the lemma.  From the remark above, we have
\begin{multline*}
\sum_{1 \leq i_1,\dotsc,i_k \leq n} \kappa_E^{(k)}[x_{i_ki_1}^{(r_1)}b_1,\dotsc,x_{i_{k-1}i_k}^{(r_k)}b_k] \otimes u_{i_1j_1}\dotsb u_{i_kj_k} \\
= \sum_{1 \leq i_1,\dotsc,i_k \leq n} \sum_{\substack{\sigma \in NC(k)\\ K(\sigma) \leq \ker \mathbf i}} c_{\sigma,\mathbf r}[b_1,\dotsc,b_k] \otimes u_{i_1j_1}\dotsb u_{i_kj_k}\\
= \sum_{\sigma \in NC(k)} c_{\sigma,\mathbf r}[b_1,\dotsc,b_k] \otimes \sum_{\substack{1 \leq i_1,\dotsc,i_k \leq n \\ K(\sigma) \leq \ker \mathbf i}} u_{i_1j_1}\dotsb u_{i_kj_k}.
\end{multline*}
Now
\begin{equation*}
 \sum_{\substack{1 \leq i_1,\dotsc,i_k \leq n \\ K(\sigma) \leq \ker \mathbf i}} u_{i_1j_1}\dotsb u_{i_kj_k} = \begin{cases} 1_{C(S_n^+)}, & K(\sigma) \leq \ker \mathbf j\\ 0, & K(\sigma) \not \leq \ker \mathbf j\end{cases},
\end{equation*}
indeed this is equivalent to the fact that $T_{K(\sigma)} \in Fix(u^{\otimes k})$ (see Section \ref{sec:easy}).  This can also be checked directly by using the relations in $C(S_n^+)$ and inducting on the number of blocks of $K(\sigma)$.  It follows that
\begin{align*}
\sum_{1 \leq i_1,\dotsc,i_k \leq n} \kappa_E^{(k)}[x_{i_ki_1}^{(r_1)}b_1,\dotsc,x_{i_{k-1}i_k}^{(r_k)}b_k] \otimes u_{i_1j_1}\dotsb u_{i_kj_k} &= \sum_{\substack{\sigma \in NC(k)\\ K(\sigma) \leq \ker \mathbf j}} c_{\sigma,\mathbf r}[b_1,\dotsc,b_k] \otimes 1_{C(S_n^+)}\\
&= \kappa_E^{(k)}[x_{j_kj_1}^{(r_1)}b_1,\dotsc,x_{j_{k-1}j_k}^{(r_k)}b_k] \otimes 1_{C(S_n^+)},
\end{align*}
so that the determining series of $X_1,\dotsc,X_s$ is invariant under quantum permutations.

Conversely, suppose that the determining series of $X_1,\dotsc,X_s$ is invariant under permutations, so that
\begin{equation*}
\kappa_E^{(k)}[x_{j_{k}j_1}^{(r_1)}b_1,\dotsc,x_{j_{k-1}j_k}^{(r_k)}b_k] \otimes 1_{C(S_n^+)} = \sum_{1 \leq i_1,\dotsc,i_k \leq n} \kappa_E^{(k)}[x_{i_ki_1}^{(r_1)}b_1,\dotsc,x_{i_{k-1}i_k}^{(r_k)}b_k] \otimes u_{i_1j_1}\dotsb u_{i_kj_k}.
\end{equation*}
Apply $(\mathrm{id} \otimes \int)$ to both sides and expand using Weingarten:
\begin{equation*}
 \kappa_E^{(k)}[x_{j_kj_1}^{(r_1)}b_1,\dotsc,x_{j_{k-1}j_k}^{(r_k)}b_k] = \sum_{\substack{\sigma,\pi \in NC(k)\\ K(\sigma) \leq \ker \mathbf j}}  W_{NC(k),n}(\pi,K(\sigma))\sum_{\substack{1 \leq i_1,\dotsc,i_k \leq n\\K(\pi) \leq \ker \mathbf i}} \kappa_E^{(k)}[x_{i_ki_1}^{(r_1)}b_1,\dotsc,x_{i_{k-1}i_k}^{(r_k)}b_k].
\end{equation*}
The result now follows by setting
\begin{equation*}
 c_{\sigma,\mathbf r}[b_1,\dotsc,b_k] = \sum_{\pi \in NC(k)} W_{NC(k),n}(\pi,K(\sigma))\sum_{\substack{1 \leq i_1,\dotsc,i_k \leq n\\K(\pi) \leq \ker \mathbf i}} \kappa_E^{(k)}[x_{i_ki_1}^{(r_1)}b_1,\dotsc,x_{i_{k-1}i_k}^{(r_k)}b_k].
\end{equation*}

\end{proof}

As discussed in Remark \ref{passage}, to prove Theorem \ref{hyperoctahedralcase} we will need to relate $NC_h(2k)$ with $NC(k)$.

\begin{lemma}\label{h+parts} If $\pi \in NC_h(2k)$, there are unique $\pi_1,\pi_2 \in NC(k)$ such that $\pi_1 \leq \pi_2$ and $\pi = \widetilde{\pi_1} \vee \widetilde{\pi_2}$.  Moreover, if $\sigma = \widetilde \sigma_1 \vee \widetilde \sigma_2$, $\pi = \widetilde \pi_1 \vee \widetilde \pi_2$ for $\sigma_1 \leq \sigma_2$ and $\pi_1 \leq \pi_2$ in $NC(k)$, then $\pi \leq \sigma$ if and only if $\sigma_1 \leq \pi_1 \leq \pi_2 \leq \sigma_2$.  In this case,
\begin{equation*}
 \mu(\pi,\sigma) = \mu(\sigma_1,\pi_1) \cdot \mu(\pi_2,\sigma_2).
\end{equation*}
\end{lemma}

\begin{proof}
Let $\pi \in NC_h(2k)$, then since each block of $\pi$ has an even number of elements we have
\begin{equation*}
 K(\pi) = \pi_1 \wr K(\pi_2)
\end{equation*}
for some $\pi_1,\pi_2 \in NC(k)$ with $\pi_1 \leq \pi_2$.  It follows from the Lemma \ref{fatfacts} that $\pi = \widetilde \pi_1 \vee \widetilde \pi_2$.  The equation above shows that $\pi_1,\pi_2$ are uniquely determined.    

Now suppose that $\sigma = \widetilde \sigma_1 \vee \widetilde \sigma_2$, $\pi = \widetilde \pi_1 \vee \widetilde \pi_2$.  Then
\begin{equation*}
 \pi \leq \sigma \Leftrightarrow K(\sigma) \leq K(\pi) \Leftrightarrow \sigma_1 \wr K(\sigma_2) \leq \pi_1 \wr K(\pi_2) \Leftrightarrow \sigma_1 \leq \pi_1 \leq \pi_2 \leq \sigma_2.
\end{equation*}
Finally, if $\sigma_1 \leq \pi_1 \leq \pi_2 \leq \sigma_2$ then
\begin{equation*}
 \mu(\pi,\sigma) = \mu(K(\sigma),K(\pi)) = \mu(\sigma_1 \wr K(\sigma_2),\pi_1\wr K(\pi_2)).
\end{equation*}
Now it is clear that the interval $[\sigma_1 \wr K(\sigma_2), \pi_1 \wr K(\pi_2)]$ in $NC(2k)$ factors as $[\sigma_1,\pi_1] \times [K(\sigma_2),K(\pi_2)]$, and so by the multiplicativity of the M\"{o}bius function we have
\begin{equation*}
 \mu(\sigma_1 \wr K(\sigma_2),\pi_1\wr K(\pi_2)) = \mu(\sigma_1,\pi_1)\cdot \mu(K(\sigma_2),K(\pi_2)) = \mu(\sigma_1, \pi_1)\cdot \mu(\pi_2,\sigma_2).
\end{equation*}

\end{proof}

We can now prove the implication $(1) \Rightarrow (2)$ in Theorem \ref{hyperoctahedralcase}.  As in the free orthogonal case, this holds for finite matrices and in a purely algebraic setting.

\begin{proposition}\label{h+finite}
Let $X_1,\dotsc,X_s$ be a family of matrices in $M_n(\mc A)$, $X_r = (x_{ij}^{(r)})_{1 \leq i,j \leq n}$.  Suppose that there is a subalgebra $1 \in \mc B \subset \mc A$ and a $\varphi$-preserving conditional expectation $E:\mc A \to \mc B$ such that $X_1,\dotsc,X_s$ is $R$-cyclic with respect to $E$, and $\Theta_X$ is invariant under quantum permutations.  Then the family $X_1,\dotsc,X_s$ is $H_n^+$-invariant.
\end{proposition}

\begin{proof}
By Lemma \ref{detserinv}, there are multilinear maps $c_{\sigma,\mathbf r}:\mc B^k \to \mc B$ for $k \in \N$, $1 \leq r_1,\dotsc,r_k \leq s$ and $\sigma \in NC(k)$ such that
\begin{equation*}
\kappa_E^{(k)}[x_{i_{11}i_{12}}^{(r_1)}b_1,\dotsc,x_{i_{k1}i_{k2}}^{(r_k)}b_k] = \sum_{\substack{\sigma \in NC(k)\\ \widetilde \sigma \vee \widetilde{1_k} \leq \ker \mathbf i}} c_{\sigma,\mathbf r}[b_1,\dotsc,b_k]
\end{equation*}
for any $b_1,\dotsc,b_k \in \mc B$ and $1 \leq i_1,\dotsc,i_k \leq n$.  For $\sigma, \pi \in NC(k)$, $\sigma \leq \pi$, define $c_{\sigma,\pi,\mathbf r}:\mc B^k \to \mc B$ recursively as follows.  If $\pi = 1_k$, $c_{\sigma,\pi,\mathbf r} = c_{\sigma,\mathbf r}$.  Otherwise let $V = \{l+1,\dotsc,l+s\}$ be an interval of $\pi$.  Let $\sigma|_V$ denote the restriction of $\sigma$ to $V$, and let $\sigma',\pi' \in NC(k-s)$ be the restrictions of $\sigma,\pi$ to $\{1,\dotsc,k\} \setminus V$. Let $\mathbf r' = r_1,\dotsc,r_l,r_{l+s+1},\dotsc,r_k$, $\mathbf r'' = r_{l+1},\dotsc,r_{l+s}$ and define
\begin{equation*}
 c_{\sigma,\pi,\mathbf r}[b_1,\dotsc,b_k] = c_{\sigma', \pi', \mathbf r'}[b_1,\dotsc,b_lc_{\sigma|_V,\mathbf r''}[b_{l+1},\dotsc,b_{l+s}],\dotsc,b_k]
\end{equation*}
for $b_1,\dotsc,b_k \in \mc B$.

Now let $\pi \in NC(k)$.  Comparing the recursive definitions of $\kappa_E^{(\pi)}$ and $c_{\sigma,\pi}$ as above, we find that
\begin{equation*}
 \kappa_E^{(\pi)}[x_{i_{11}i_{12}}^{(r_1)}b_1,\dotsc,x_{i_{k1}i_{k2}}^{(r_k)}b_k] = \sum_{\substack{\sigma\in NC(k)\\ \sigma \leq \pi\\ \widetilde \sigma \vee \widetilde \pi \leq \ker \mathbf i}} c_{\sigma,\pi,\mathbf r}[b_1,\dotsc,b_k].
\end{equation*}

If $\tau \in NC_h(2k)$, use Lemma \ref{h+parts} to find unique $\sigma,\pi \in NC(k)$ with $\sigma \leq \pi$ and $\widetilde \sigma \vee \widetilde \pi = \tau$, and define $c_{\tau,\mathbf r} = \varphi(c_{\sigma,\pi, \mathbf r}[1,\dotsc,1])$.  We then have
\begin{align*}
 \varphi(x_{i_{11}i_{12}}^{(r_1)}\dotsb x_{i_{k1}i_{k2}}^{(r_k)}) &= \sum_{\pi \in NC(k)} \varphi(\kappa_E^{(\pi)}[x_{i_{11}i_{12}}^{(r_1)},\dotsc,x_{i_{k1}i_{k2}}^{(r_k)}])\\
&= \sum_{\substack{\sigma, \pi \in NC(k)\\ \sigma \leq \pi\\ \widetilde \sigma \vee \widetilde \pi \leq \ker \mathbf i}} \varphi(c_{\sigma,\pi,\mathbf r}[1,\dotsc,1])\\
&= \sum_{\substack{\tau \in NC_h(2k)\\ \tau \leq \ker \mathbf i}} c_{\tau,\mathbf r},
\end{align*}
and the result follows from the characterization of $H_n^+$-invariant families in Theorem \ref{easyinv}.
\end{proof}

We will now complete the proof of Theorem \ref{hyperoctahedralcase} by showing $(1) \Rightarrow (2)$.

\begin{proof}[Proof of Theorem \ref{hyperoctahedralcase}]
Let $\mc B$ denote the $H^+$-invariant subalgebra introduced in Section \ref{sec:qinv}.  Let $b_1,\dotsc,b_k \in \mc B$, $1 \leq r_1,\dotsc,r_k \leq s$ and $j_{11},\dotsc,j_{k2} \in \N$, then from Theorem \ref{limitcumulants} we have
\begin{multline*}
\kappa_E^{(k)}[x_{j_{11}j_{12}}^{(r_1)}b_1,\dotsc,x_{j_{k1}j_{k2}}^{(r_k)}b_k] \\
= \lim_{n \to \infty} \sum_{\substack{\sigma \in NC_h(2k)\\ \sigma \leq \ker \mathbf j\\ \sigma \vee \hat 0_k = 1_{2k}}} \sum_{\substack{\pi \in NC_h(2k)\\ \pi \leq \sigma}} \mu(\pi,\sigma)n^{-|\pi|}\sum_{\substack{1 \leq i_{11},\dotsc,i_{k2} \leq n\\ \pi \leq \ker \mathbf i}} x_{i_{11}i_{12}}^{(r_1)}\dotsb x_{i_{k1}i_{k2}}^{(r_k)}b_k.
\end{multline*}
Now use Lemma \ref{h+parts} to replace $\sigma, \pi \in NC_h(2k)$ in the equation above by $\widetilde \sigma_1 \vee \widetilde \sigma_2$ and $\widetilde \pi_1 \vee \widetilde \pi_2$, where $\sigma_1 \leq \pi_1 \leq \pi_2 \leq \sigma_2$ are in $NC(k)$.  Note that the condition $\widetilde \sigma_1 \vee \widetilde \sigma_2 \vee \hat 0_k = 1_{2k}$ forces $\sigma_2 = 1_k$.  But then $\widetilde \sigma_2 \leq \ker \mathbf j$ forces the $R$-cyclicity condition with respect to $E$.

It remains only to show that the determining series $\Theta_X$ is invariant under quantum permutations.  From the previous paragraph, we have
\begin{multline*}
 \kappa_E^{(k)}[x_{j_{11}j_{12}}^{(r_1)}b_1,\dotsc,x_{j_{k1}j_{k2}}^{(r_k)}b_k] \\
= \lim_{n \to \infty} \sum_{\substack{\sigma \in NC(k)\\ \widetilde \sigma \vee \widetilde{1_k} \leq \ker \mathbf j}} \sum_{\substack{\pi_1,\pi_2 \in NC(k)\\ \sigma \leq \pi_1 \leq \pi_2}} \mu(\sigma,\pi_1)\mu(\pi_2,1_k)n^{-|\widetilde \pi_1 \vee \widetilde \pi_2|}\sum_{\substack{1 \leq i_{11},\dotsc,i_{k2} \leq n\\ \widetilde \pi_1 \vee \widetilde \pi_2 \leq \ker \mathbf i}} x_{i_{11}i_{12}}^{(r_1)}b_1\dotsb x_{i_{k1}i_{k2}}^{(r_k)}b_k.
\end{multline*}

Now we would like to define
\begin{equation*}
c_{\sigma,\mathbf r}[b_1,\dotsc,b_k] = \lim_{n \to \infty}\sum_{\substack{\pi_1,\pi_2 \in NC(k)\\ \sigma \leq \pi_1 \leq \pi_2}} \mu(\sigma,\pi_1)\mu(\pi_2,1_k)n^{-|\widetilde \pi_1 \vee \widetilde \pi_2|}\sum_{\substack{1 \leq i_{11},\dotsc,i_{k2} \leq n\\ \widetilde \pi_1 \vee \widetilde \pi_2 \leq \ker \mathbf i}} x_{i_{11}i_{12}}^{(r_1)}b_1\dotsb x_{i_{k1}i_{k2}}^{(r_k)}b_k,
\end{equation*}
and the result would follow from Lemma \ref{detserinv}.  However we must check that the right hand side converges.  Let $c^n_{\sigma,\mathbf r}[b_1,\dotsc,b_k]$ denote the right hand side, then from the above paragraph we know that for any $\tau \in NC(k)$,
\begin{equation*}
 \sum_{\substack{\pi \in NC(k)\\ \pi \leq \tau}} c^n_{\pi,\mathbf r}[b_1,\dotsc,b_k] 
\end{equation*}
converges to $\kappa_E^{(k)}[x_{j_{11}j_{12}}^{(r_1)}b_1,\dotsc,x_{j_{k1}j_{k2}}^{(r_k)}b_k]$ for any $j_{11},\dotsc,j_{k2}$ such that $\ker \mathbf j = \widetilde \tau \vee \widetilde 1_k$.  But then
\begin{equation*}
c^n_{\sigma,\mathbf r}[b_1,\dotsc,b_k] = \sum_{\substack{\tau \in NC(k)\\ \tau \leq \sigma}} \mu(\tau,\sigma) \sum_{\substack{\pi \in NC(k)\\ \pi \leq \tau}} c^n_{\pi,\mathbf r}[b_1,\dotsc,b_k]
\end{equation*}
converges as well, which completes the proof.
\end{proof}

\section{Concluding remarks}\label{sec:conclusion}

In this paper we have used the framework of ``free'' quantum groups from \cite{bsp} to study families of infinite matrices of random variables whose joint distribution is invariant under conjugation by a compact orthogonal quantum group.  In particular, we have given complete characterizations of the families which are invariant under conjugation by $O_n^+$ or $H_n^+$.  

A remaining question is to better understand the structure of matrices which are invariant under conjugation by $S_n^+$ (or $B_n^+$).  As mentioned in the introduction, one surprise here is that self-adjoint matrices with freely independent and identically distributed entries above the diagonal are not necessarily $S_n^+$-invariant, as we now show.

\begin{proposition}
Let $(x_{ij})_{1 \leq i \leq j \leq n}$ be a family of freely independent $(0,1)$-semicircular random variables in a noncommutative probability space $(\mc A, \varphi)$, and let $x_{ij} = x_{ji}$ for $i > j$.  Let $X = (x_{ij})_{1 \leq i,j \leq n}$ in $M_n(\mc A)$.  If $n \geq 4$, then $X$ is not $S_n^+$-invariant.
\end{proposition}

\begin{proof}
Let $\pi,\sigma,\tau \in \mc P(4)$ be the partitions
\begin{center}
\begin{pspicture}(-1.5,0)(4,2)
{\psset{xunit=1cm,yunit=.6cm,linewidth=.5pt}
\psline(0,2)(0,0)\psline(0,0)(2,0)\psline(2,0)(2,2)
\psline(1,2)(1,1)\psline(1,1)(1.9,1)\psline(2.1,1)(3,1)\psline(3,1)(3,2)
\uput[u](0,2){1}\uput[u](1,2){2}\uput[u](2,2){3}
\uput[u](3,2){4} \uput[u](-.6,1){$\pi=$}}
\end{pspicture}
\qquad
\begin{pspicture}(0,0)(4,2)
{\psset{xunit=1cm,yunit=.6cm,linewidth=.5pt} 
\psline(0,2)(0,0)\psline(0,0)(3,0)\psline(3,0)(3,2)
\psline(1,2)(1,1)\psline(1,1)(2,1)\psline(2,1)(2,2)
\uput[u](0,2){1}\uput[u](1,2){2}\uput[u](2,2){3}
\uput[u](3,2){4} \uput[u](-.6,1){$\sigma=$}}
\end{pspicture}
\qquad
\begin{pspicture}(0,0)(4,2)
{\psset{xunit=1cm,yunit=.6cm,linewidth=.5pt} 
\psline(0,2)(0,0)\psline(0,0)(3,0)\psline(3,0)(3,2)
\psline(1,0)(1,2)\psline(2,0)(2,2)
\uput[u](0,2){1}\uput[u](1,2){2}\uput[u](2,2){3}
\uput[u](3,2){4} \uput[u](-.6,1){$\tau=$}}
\end{pspicture}
\end{center}
Observe that $\varphi(x_{i_1i_2}x_{i_3i_4}) = \delta_{\pi}(\mathbf i) + \delta_{\sigma}(\mathbf i) - \delta_{\tau}(\mathbf i)$.

Suppose that $X$ were $S_n^+$-invariant, then we would have the equality
\begin{equation*}
\sum_{1 \leq i_1,i_2,i_3,i_4 \leq n} \varphi(x_{i_1i_2}x_{i_3i_4})u_{i_1j_1}u_{i_2j_2}u_{i_3j_3}u_{i_4j_4} = \varphi(x_{j_1j_2}x_{j_3j_4})\cdot 1_{C_(S_n^+)}
\end{equation*}
for any $1 \leq j_1,\dotsc, j_4 \leq n$.  In other words,
\begin{equation*}
 \sum_{1 \leq i_1,i_2,i_3,i_4 \leq n} (\delta_{\pi}(\mathbf i) + \delta_{\sigma}(\mathbf i) - \delta_{\tau}(\mathbf i)) u_{i_1j_1}u_{i_2j_2}u_{i_3j_3}u_{i_4j_4} = (\delta_{\pi}(\mathbf j) + \delta_{\sigma}(\mathbf j) - \delta_{\tau}(\mathbf j))\cdot 1_{C(S_n^+)}.
\end{equation*}

Recall from Section \ref{sec:easy} that associated to any $\nu \in \mc P(4)$ there is the vector
\begin{equation*}
 T_{\nu} = \sum_{1 \leq i_1,i_2,i_3,i_4 \leq n} \delta_{\nu}(\mathbf i) e_{i_1} \otimes e_{i_2} \otimes e_{i_3} \otimes e_{i_4} \in (\C^n)^{\otimes 4}.
\end{equation*}
Let $\Psi:C(S_n^+) \to C(S_n^+)$ be the automorphism $\Psi(f) = S(f)^*$.  Then we have
\begin{equation*}
 (\mathrm{id} \otimes \Psi)u^{\otimes 4}(T_\nu) = \negthickspace\negthickspace \sum_{1 \leq j_1,\dotsc,j_4 \leq n} e_{j_1} \otimes e_{j_2} \otimes e_{j_3} \otimes e_{j_4} \otimes \bigl(\sum_{1 \leq i_1,i_2,i_3,i_4 \leq n}\delta_{\nu}(\mathbf i) u_{i_1j_1}u_{i_2j_2}u_{i_3j_3}u_{i_4j_4}\bigr)
\end{equation*}
Since $\Psi$ is an automorphism we have
\begin{equation*}
 T_{\nu} \in Fix(u^{\otimes 4}) \Leftrightarrow \sum_{1 \leq i_1,i_2,i_3,i_4 \leq n} \delta_{\nu}(\mathbf i) u_{i_1j_1}u_{i_2j_2}u_{i_3j_3}u_{i_4j_4} = \delta_{\nu}(\mathbf j) \cdot 1_{C(S_n^+)}.
\end{equation*}
Now since $\sigma$ and $\tau$ are non-crossing, we know from Section \ref{sec:easy} that $T_\sigma$ and $T_\tau$ are in $Fix(u^{\otimes 4})$.  It then follows from the equation above that also $T_\pi \in Fix(u^{\otimes 4})$, which is known from \cite{ban} to be false.
\end{proof}
We remark that one may easily modify the proof to show that if $(x_{ij})_{1 \leq i,j\leq n}$ are freely independent $(0,1)$-semicircular random variables then the (non self-adjoint) matrix $X = (x_{ij})_{1 \leq i,j \leq n}$ is also not $S^+$-invariant.  

Let us consider now an infinite self-adjoint $O^+$-invariant matrix $X = (x_{ij})_{1 \leq i,j \leq n}$.  Assume for simplicity that the $O^+$-invariant subalgebra (see \ref{sec:infinvar}) is equal to $\C$ (in the classical setting such a matrix is called \textit{dissociated}).  By Theorem 1, $X$ is uniformly $R$-cyclic, so that the joint distribution of $(x_{ij})_{i,j \in \N}$ is determined by that of $x_{11}$.  By Theorem 2 the distribution of $x_{11}$ is freely infinitely divisible.  It follows that the distribution of $x_{11}$ is the weak limit as $k \to \infty$ of compound Poisson distributions of the form $s^{(k)}a^{(k)}s^{(k)}$, where $s^{(k)}$ is a centered semicircular random variable which is freely independent from $a^{(k)}$ (see e.g. \cite{ns}).  Therefore the joint distribution of $(x_{ij})_{i,j \in \N}$ is the weak limit as $k \to \infty$ of the joint distribution of $(s_i^{(k)}a^{(k)}s_j^{(k)})_{i,j \in \N}$, where for each $k$, $(s_i^{(k)})_{i \in \N}$ is a sequence of freely independent centered semicircular random variables, with the same variance as $s^{(k)}$, which is freely independent from $a^{(k)}$.  Since free and identically distributed centered semicircular sequences are characterized by $O^+$-invariance (\cite{cur4},\cite{bcs2}), we find that (dissociated) self-adjoint $O^+$-invariant matrices can be obtained as limits of products of $O^+$-invariance sequences.

In general, if $(y_i)_{i \in \N}$ is an infinite $G$-invariant sequence of self-adjoint random variables, where $G$ is one of $O^+,S^+,H^+$ or $B^+$, and $a$ is freely independent from $\{y_i:i \in \N\}$, then one can show that $X = (x_{ij})_{i,j \in \N}$, $x_{ij} = y_iay_j$, is a self-adjoint $G$-invariant matrix.  In view of the free orthogonal case discussed above, it is tempting to conjecture that any self-adjoint $G^+$-invariant matrix may be obtained as a weak limit of matrices of this form.   However this does not seem to be clear, even in the case $G = H^+$.

One may also consider $k$-dimensional $G$-invariant arrays $X = (x_{i_1,\dotsc,i_k})_{i_1,\dotsc,i_k \in \N}$, $k \geq 3$.  In the classical setting, for $G = O$ or $S$, Kallenberg has given a uniform treatment for all values of $k$, see \cite{kal}.  In the free setting it is not clear how to deal with the case $k \geq 3$, as our characterization uses heavily the matricial structure for $k = 2$.  However, in light of the discussion above one might suspect that any infinite $G$-invariant array might be obtained from products of a $G$-invariant sequences, at least in case $G = O^+$. 

Let us point out a relation between Theorem 1 and our recent paper \cite{cs1}.  Our main result there is a statement of asymptotic freeness between constant operator-valued matrices and free unitary or orthogonal matrices.  In particular we show that if $A_N$ and $B_N$ are constant operator-valued $N \times N$ matrices with limiting distributions as $N \to \infty$ and, for each $N$, $U_N$ is a Haar-distributed free orthogonal $N\times N$ matrix, then $U_NA_NU_N^*$ and $B_N$ are asymptotically free with amalgamation as $N \to \infty$.  Now suppose that $X = (x_{ij})_{i,j \in \N}$ is an infinite self-adjoint $O^+$-invariant matrix, and for each $N$ let $X_N = (x_{ij})_{1 \leq i,j \leq N}$.  Let $B_N$ be a sequence of matrices in $M_N(\mc B)$ which has a limiting $\mc B$-valued distribution as $N \to \infty$.  Since the entries of $U_NX_NU_N^*$ have the same joint distribution as the entries of $X_N$ by assumption, our result suggests that $X_N$ and $B_N$ should be asymptotically free with amalgamation over $\mc B$. By Theorem \ref{uniformrcyclictheorem}, this would suggest that $X_N$ should be ``asymptotically'' uniformly $R$-cyclic with respect to the expectation onto $\mc B$.  But the definition of uniform $R$-cyclicity in terms of $\mc B$-valued cumulants would suggest that $X_N$ should then be uniformly $R$-cyclic for each $N \in \N$, which is the content of Theorem 1.  Of course there are many difficulties in making such an argument rigorous.  But let us remark that one may indeed adapt the methods from \cite{cs1} to give another proof of Theorem 1, by showing that if $X_1,\dotsc,X_s$ is a self-adjoint $O^+$-invariant family of infinite matrices with entries in $(M,\varphi)$, then for each $N$ the family $\{X_1^{(N)},\dotsc,X_s^{(N)}\} \subset M_N(M)$ is free from $M_N(\mc B)$ with amalgamation over $\mc B$.  We have chosen instead to work with the $\mc B$-valued cumulants of the entries in this paper, as it allows a more uniform treatment for free quantum groups.

Finally, let us discuss the situation for \textit{separately invariant} matrices, i.e. matrices $X = (x_{ij})_{1 \leq i,j \leq n}$ which are invariant under multiplication on the left or right by matrices in the compact orthogonal (quantum) group $G$.  The Aldous-Hoover characterization of jointly exchangeable arrays also holds, with slight modifications, for separately exchangeable arrays \cite{ald},\cite{hoover}.  While the notion of separate invariance makes perfect sense for free quantum groups, it does not appear to lead to interesting results in free probability.  Indeed if $(x_i)_{i \in \N}$ and $(y_i)_{i \in \N}$ are both free and identically distributed sequences, but which are stochastically independent from each other, then the matrix $X = (x_{ij})_{i,j \in \N}$, $x_{ij} = x_iy_j$, is separately $S^+$-invariant.  This is somewhat reminiscent of the situation for exchangeable sequences of noncommutative random variables, see \cite{kos}.  

\def\cprime{$'$}

\end{document}